\documentclass[a4paper,11pt,onecolumn,twoside]{article}
\usepackage{tikz}
\usetikzlibrary{arrows.meta}
\usepackage{fancyhdr}
\usepackage{hyperref}
\usepackage[utf8]{inputenc}
\usepackage{amsmath,amsfonts,amssymb}
\usepackage{graphicx}
\usepackage{mathptmx}
\DeclareUnicodeCharacter{FB01}{fi}
\usepackage{amsthm}
\usepackage{booktabs}
\usepackage{amssymb}
\usepackage[mathscr]{eucal}
\usepackage[labelfont=bf]{caption}
\usepackage{indentfirst}
\usepackage{caption}
\usepackage{enumitem}
\usepackage{subfigure}
\usepackage{authblk}
\usepackage[all]{xy}
\usepackage{geometry}
\usepackage{mathrsfs}
\usepackage{newtxmath}
\usepackage{hyperref}
\numberwithin{equation}{section}

\newtheorem{thm}{Theorem}[section]

\newtheorem{definition}[thm]{Definition}
\newtheorem{exm}[thm]{Example}

\newtheorem{lemma}[thm]{Lemma}
\newtheorem{cor}[thm]{Corollary}
\newtheorem{prop}[thm]{Proposition}

\newtheorem{rmk}[thm]{Remark}

\newcommand{\GG}{{^{G}_{G}\mathcal{YD}^{\omega}}}

\newcommand{\HH}{{^{H}_{H}\mathcal{YD}}}

\newcommand{\II}{{^{I}_{I}\mathcal{YD}}}
\newcommand{\Dff}{{^{D_8}_{D_8}\mathcal{YD}_{\operatorname{fd}}}}
\newcommand{\CC}{\mathbb{C}^*}
\newcommand{\MC}{\mathscr{C}}
\newcommand{\MM}{\mathscr{M}}
\newcommand{\stg}{\sigma_{\tau}(1,g)}
\newcommand{\stx}{\sigma_{\tau}(\chi,1)}
\newcommand{\DG}{D^{\omega}(G)}
\newcommand{\GO}{\Gamma^{\omega}}
\newcommand{\VG}{\operatorname{Vec}_G^{\omega}}
\newcommand{\bVG}{\textbf{Vec}_G^{\omega}}
\newcommand{\stxg}{\sigma_{\tau}(\chi,g^b)}
\newcommand{\GGd}{{^{G}_{G}\mathcal{YD}^{\omega}_{\operatorname{fd}}}}
\newcommand{\HHd}{{^{H}_{H}\mathcal{YD}_{\operatorname{fd}}}}
\title{\textbf{ On Gauge Equivalence of Twisted Quantum Doubles }}

\author{BOWEN LI AND GONGXIANG LIU}

\date{}

\setlist{nolistsep}
\begin{document}\large
	\maketitle

	\setlength{\oddsidemargin}{ -1cm}
	\setlength{\evensidemargin}{\oddsidemargin}
	\setlength{\textwidth}{15.50cm}
	\vspace{-.8cm}
	
	\setcounter{page}{1}
	
	\setlength{\oddsidemargin}{-.6cm}  % 3.17cm - 1 inch
	\setlength{\evensidemargin}{\oddsidemargin}
	\setlength{\textwidth}{17.00cm}
 	\thispagestyle{fancy}
	\fancyhf{}
	\fancyfoot[R]{\thepage}
	\fancyfoot[L]{Supported by NSFC 12271243}
	\fancyhead{} % 页眉清空
	\renewcommand{\footrulewidth}{1pt}
	\renewcommand{\headrulewidth}{0pt}
	\begin{abstract}
		We study the quantum double of a ﬁnite abelian group $G$ twisted by a $3$-cocycle and give a sufficient condition when such a twisted quantum double  will  be gauge equivalent to a ordinary quantum double of a finite group. Moreover, we will determine when a twisted quantum double of a cyclic group is genuine. As an application, we contribute to the classification of coradically graded finite-dimensional pointed coquasi-Hopf algebras over  abelian groups. As a byproduct, we show that the Nichols algebras $\mathcal{B}(M_1\oplus M_2 \oplus M_3)$ are infinite-dimensional where $M_1,M_2,M_3$ are three different simple Yetter-Drinfeld modules of $D_8$.
	\end{abstract}
	\section{Introduction}
Given a finite group $G$ and a normalized $3$-cocycle $\omega \in Z^3(G,\mathbb{C}^*)$, Dijkgraaf-Pasquier-Roche defines a certain braided quasi-Hopf algebra(twisted quantum double) $D^{\omega}(G)$ in \cite{DPR}. This article is aimed to study the gauge equivalence  between certain twisted quantum doubles, which leads to tensor equivalence between their representation categories. \par 
We  review the background motivation of $D^{\omega}(G)$ briefly. Although this concept and our results are purely algebraic, the motivation of studying these problems comes from conformal field theory and vertex operator algebra(cf.\cite{Kac98},\cite{Bor86},\cite{FHL93}). The readers only need to understand that vertex operator algebras have a representation theory, especially, in \cite{Huang}, he proved that  if $\mathbb{V}$ is $C_2$-cofinite, rational, CFT type (i.e. $V(1)=\mathbb{C}1$), and self dual (i.e. $\mathbb{V}\cong \mathbb{V}'$), then $\operatorname{Rep}(\mathbb{V})$ is a modular tensor category. Then by reconstruction theory \cite{Majid}, there might be a weak quasi-Hopf algebra $H$ with the property that $\operatorname{Rep}(H) \cong \operatorname{Rep}(\mathbb{V})$ as modular  category. In the context of \cite{DPR}, the authors conjectured that one can take $H$ to be a twisted quantum double $D^{\omega}(G)$ of $G$ in the case when $\mathbb{V}$ is a so-called holomorphic orbifold model, that is there is a simple vertex operator algebra $\mathbb{W}$ and a ﬁnite group of automorphisms $G$ of $\mathbb{W}$ such that $\mathbb{V}=\mathbb{W}^G$, see also \cite{MG}. 

 Now suppose there's a equivalence of braided tensor category $\operatorname{Rep}(\mathbb{W}_1^{G_1}) \cong \operatorname{Rep}(\mathbb{W}_2^{G_2})$, for two holomorphic vertex operator algebras and finite groups $G_1, G_2$. If the conjecture is true, then there must be a braided tensor equivalence:
\begin{align}\label{(1.1)}
	\operatorname{Rep}(D^{\omega_1}(G_1))\cong \operatorname{Rep}(D^{\omega_2}(G_2))
\end{align} 
for some choices of 3-cocycles $\omega_1, \omega_2$ of $G_1, G_2$.
Conversely, deciding when equivalences such as (\ref{(1.1)}) can hold gives information about the vertex operator algebras and the cocycles that they determine. This is an interesting problem in its own right, and is the one we consider here. 
 
The case in which the two twisted quantum doubles in question are commutative, i.e. two groups are abelian and the $3$-cocycles are abelian $3$-cocycle was solved in \cite{Abcocycle}. In \cite{elementary}, the authors dealt with the case when $G_1$  is an elementary abelian 2-group and $G_2$ turns out to be an extra-special $2$-group. Here we are concerned with a particular case of (\ref{(1.1)}) when taking $G_1$ as a finite abelian group, $\omega_1$ is an arbitrary normalized $3$-cocycle and $G_2$ is a finite group, $\omega_2$ is trivial:
\begin{align}\label{(1.2)}
	\operatorname{Rep}(D^{\omega_1}(G_1))\cong \operatorname{Rep}(D(G_2)).
\end{align} 
We will give a  sufficient condition when equivalence (\ref{(1.2)}) holds and the reason we can do it is that we knew an explicit expression of these $3$-cocycles as indicated below. 

Let $G$ be a finite abelian group which is isomorphic to $Z_{m_1}\times Z_{m_2}\times \cdots \times Z_{m_n}$ with $m_i \mid m_{i+1}$ for $1 \leq i \leq n-1$. Thanks to \cite{grcate} and \cite{QQG}, we can write down all representatives of normalized $3$-cocycles on $G$: 
\begin{align}\label{(1.3)}
	& \omega\left( g_1^{i_1} \cdots g_n^{i_n}, g_1^{j_1} \cdots g_n^{j_n}, g_1^{k_1} \cdots g_n^{k_n}\right)  = \prod_{l=1}^n \zeta_{m_l}^{a_l i_l\left[\frac{j_l+k_l}{m_l}\right]} \prod_{1 \leq s<t \leq n} \zeta_{m_s}^{a_{s t} k_s\left[\frac{i_t+j_t}{m_t}\right]} \prod_{1 \leq r<s<t \leq n} \zeta_{\left(m_r, m_s, m_t\right)}^{a_{r s t} k_r j_s i_t}
\end{align}
where $0 \leq a_l <m_l$, $0\leq a_{st} <(m_s,m_t)$, $0\leq a_{rst} <(m_r,m_s,m_t)$. Let \begin{equation}
\underline{a}=(a_1,a_2,...,a_l,...,a_n,a_{12},a_{13},...,a_{st},...,a_{n-1,n},a_{123},...,a_{rst},...a_{n-2,n-1,n}).
\end{equation}For a fixed $\underline{a}$, we define the following sets:
\begin{equation}
	\begin{aligned}
			&A_1:=\left\lbrace i | a_{ij}\neq 0, 1\leq i <j \leq n \right\rbrace, \qquad A_2:=\left\lbrace i | a_{ijk}\neq 0, 1\leq i <j<k \leq n \right\rbrace, \\
		&B_1:= \left\lbrace j | a_{ij}\neq 0, 1\leq i <j \leq n \right\rbrace, \qquad B_2:=\left\lbrace j,k | a_{ijk}\neq 0, 1\leq i <j<k \leq n \right\rbrace. 
	\end{aligned}
\end{equation}
Let $A=A_1 \cup A_2$, $B=B_1 \cup B_2$.  The first main result of the paper is the following one.
\begin{thm}\label{thm1.1}
	Let $G$ be a finite abelian group and $\omega$  a normalized $3$-cocycle on $G$ as in (\ref{(1.3)}). If  the following condition holds:  \begin{itemize}
	    \item[\emph{(i)}] $a_i = 0$ for all $ 1\leq i \leq n$.
     \item[\emph{(ii)}] $A \cap B = \emptyset$.
	\end{itemize}
 Then $D^{\omega}(G)$  will  be gauge equivalent to $D(G')$ for a finite group $G'$.
\end{thm}

%The trick of  proof to this theorem is to apply the theory of categorical Morita equivalence, which arose in \cite{Muger},\cite{FJRS}. It is an equivalent relation for tensor categories. An important result with respect to the theory of categorical Morita equivalence is that two fusion categories are categorical Morita equivalent if and only if  their center are braided  equivalent \cite{ENO}. On the other hand,
%in \cite{center}, the authors showed that $\mathcal{Z}(\operatorname{Vec}_G^{\omega})$ is  tensor equivalent to $\operatorname{Rep}(D^{\omega}(G))$. Hence we can reduce the problem to when $\operatorname{Vec}_{G_1}^{\omega_1}$ and $\operatorname{Vec}_{G_2}^{\omega_2}$ are categorical Morita equivalent. The general case of the above question is answered by  \cite{Naidu} and \cite{Morita} with certain conditions. Nevertheless, explicit representatives of $3$-cocycles are not involved in their work. Based on their results, we continue to investigate when $\operatorname{Vec}_{G}^{\omega}$ and $\operatorname{Vec}_{G'}$ are categorical Morita equivalent more explicitly. Once we give a sufficient condition when $\operatorname{Vec}_{G}^{\omega}$ and $\operatorname{Vec}_{G'}$ are categorical Morita equivalent. Then we know when $\DG$ is gauge equivalent to $D(G')$ for a finite group $G'$. This is the  approach to prove Theorem \ref{thm1.1}.

Recall that for a quasi-Hopf algebra $H$, we say $H$ is genuine if it will never be gauge equivalent to a Hopf algebra. 
Obviously, if $D^{\omega}(G)$  is gauge equivalent to $D(G')$ for a finite group $G'$, then $\DG$ isn't genuine. One may ask, if $D^{\omega}(G)$  will never be gauge equivalent to $D(G')$ for a finite group $G'$, then whether $\DG$ is genuine or not. In general, the answer is no. In fact, Theorem 9.4 in  \cite{Abcocycle} tells us if $G$ is of odd order, then $\DG$ is not genuine. But, we will prove for arbitrary finite cyclic group $G$ with a nontrivial $3$-cocycle, the $D^{\omega}(G)$  will never be gauge equivalent to $D(G')$ for a finite group $G'$. Hence, to study the genuineness of a twisted quantum double is another question. Here is some results about this question up to the knowledge of the authors.  Example 9.5 in  \cite{Abcocycle} gives us the first example of genuine twisted quantum, say $D^{\omega}(Z_2)$, where $\omega$ is the nontrivial $3$-cocycle on $Z_2$. In \cite{indicator} Theorem 4.1, the authors showed that if $G$ is abelian, and $\omega$ is an abelian cocycle, then $\DG$ is genuine if and only if there exists $ V \in \operatorname{Rep}(\DG)$ such that $\overline{\nu}(V)=0$, where $\overline{\nu}$ is the total Frobenius-Schur indicator of $\operatorname{Rep}(\DG)$. Let $G$ be a finite cyclic group and $\omega$ a nontrivial $3$-cocycle on $G$. Our second main result is to provide a discriminant method for whether $\DG$ is genuine or not, also though using the explicit expression of $3$-cocycles. Here is the result.
\begin{thm}\label{thm 3.19}
	Let $G \cong Z_m$ be a finite cyclic group and $\omega(g^i,g^j,g^k)= \zeta_m^{ai[\frac{j+k}{m}]}$  for $ 1\leq a < m$. Then $\DG$ is genuine if and only if $(m,2a) \nmid (m,a)$.
\end{thm}
As an application, we apply Theorem $\ref{(1.1)}$ to the classification of pointed finite-dimensional coquasi-Hopf algebras, which has been  investigated in \cite{rank1},\cite{coquasitriangular},\cite{tame-rep},\cite{FQQR2},\cite{QQG}. 
%Similar to the Hopf case, any coradically graded pointed coquasi-Hopf algebra will be the quasi-version bosonization of a twisted Nichols algebra in twisted Yetter-Drinfeld module category and a group algebra.
Recently, the classification of coradically graded finite-dimensional coquasi-Hopf algebras over abelian groups has been  done in \cite{huang2024classification}.  One of key ingredients in this paper is Proposition 4.1. The proof of this proposition is rather technical and depends on complicated and long computations.  Here we use our method to give a simple proof (see Subsection 5.4 for related illustrations). 
\begin{thm}\label{thm1.3}
Let $G \cong Z_2 \times Z_2 \times Z_2=\left\langle g_1\right\rangle  \times \left\langle g_2\right\rangle  \times \left\langle g_3\right\rangle $ be an abelian group and $\omega$  the $3$-cocycle on $G$:
\begin{align}\label{1.3}
	& \omega\left( g_1^{i_1}  g_2^{i_2}g_3^{i_3}, g_1^{j_1} g_2^{j_2} g_3^{j_3}, g_1^{k_1} g_2^{k_2}g_3^{k_3}\right)  = (-1)^{k_1 j_2 i_3}.
\end{align}
 Let $V_1, V_2, V_3 \in  \GG$ be simple twisted Yetter-Drinfeld modules such that $\operatorname{dim}(V_i)=2$,  $\operatorname{deg}(V_i)=g_i$, $1 \leq i \leq 3$, such that $g_i \rhd v= -v$ for all $v \in V_i$, $1\leq i \leq 3$. Then the Nichols algebra $B(V_1 \oplus V_2 \oplus V_3 )$ is inﬁnite-dimensional.
\end{thm}
In order to apply our method to show above theorem, we also proved that the Nichols algebras $\mathcal{B}(M_1\oplus M_2 \oplus M_3)$ are always infinite-dimensional where $M_1,M_2,M_3$ are three different simple Yetter-Drinfeld modules of $D_8$.
%It should be noted  that $\operatorname{deg}(V_i)=g_i$ means  $V_i={^{g_i}V_i}$, where $ ^gV_i:=\left\lbrace v \in V\mid \delta(v)=g\otimes v\right\rbrace$ and $\delta$ is the comodule structure of $V_i$. Since $V_1$, $V_2$ and $V_3$ has different degree, they are pairwise nonisomorphic simple Yetter-Drinfeld modules.\par 
 %The methods to prove the theorem are as follows, in \cite{YD-module}, the author suggested twisted Yetter-Drinfeld module category over a group $G$ with a $3$-cocycle $\omega$ can be regarded as the center of  $\textbf{Vec}_G^{\omega}$. The theory of categorical Morita equivalence is employed here. We may turn the study of  twisted Nichols algebra in twisted Yetter-Drinfeld module category to the study of usual Nichols algebra in Yetter-Drinfeld module category such that we may apply the beautiful  classification result of finite-dimensional Nichols algebra over a finite group. For further knowledge, see \cite{Angiono},\cite{AHS},\cite{H1} and \cite{Hec}.  In particular, Heckenbeger has classified all finite-dimensional Nichols algebra over a non-abelian group of rank $\geq 3$ in \cite{Hec}. Finally, we may show such a  twisted nondiagonal Nichols algebra, hence the induced coquasi-Hopf algebra is infinite-dimensional.\par 

Here is the layout of the paper. Section $2$ is devoted to some preliminary materials. In Section $3$, we  provide a sufficient condition that when  the twisted quantum double of a finite abelian group will  be gauge equivalent to the ordinary quantum double of a finite group. In Section $4$, we give a criterion when a twisted quantum double of a finite cyclic group is genuine.  Section $5$ is concentrated on above application.
%, based on the classification result of finite-dimensional Nichols algebra over dihedral group $D_4$, we give an application to the classification of finite-dimensional coquasi-Hopf algebra. 	\par 

Throughout this paper, $\vmathbb{k}$ is an algebraically closed ﬁeld with characteristic zero and all linear spaces are over $\vmathbb{k}$. $\textbf{Vec}_G^{\omega}$ is the tensor category of $G$-graded vector spaces with associativity defined by $\omega$ and $\VG$ is the fusion category of finite dimensional $G$-graded vector spaces with associativity defined by $\omega$. $\textbf{Rep}(D^{\omega}(G))$ is the tensor category of representations of $D^{\omega}(G)$ while $\operatorname{Rep}(D^{\omega}(G))$ is the fusion category of finite-dimensional representations of $D^{\omega}(G)$.
\section{Preliminaries}
Here we recall some necessary notions and results.  
\subsection{$3$-cocycles of finite abelian groups}
 %The problem of giving a complete list of the representatives of the $3$-cohomology classes in $H^3(G,
%\mathbb{C}^{\times})$ for all finite abelian groups has been solved in \cite{QQG}, which can be described as follows: \par 
By Fundamental theorem of finite abelian groups, any finite abelian group is of the form: $Z_{m_1}\times Z_{m_2}\cdots \times Z_{m_n}$ with $m_i \mid m_{i+1}$ for $1 \leq i \leq n-1$. Denote $\mathscr{A}$ the the set of all $\mathbb{N}$-sequences:
$$\underline{a}=(a_1,a_2,...,a_l,...,a_n,a_{12},a_{13},...,a_{st},...,a_{n-1,n},a_{123},...,a_{rst},...a_{n-2,n-1,n})$$
such that $0 \leq a_l <m_l$, $0\leq a_{st} <(m_s,m_t)$, $0\leq a_{rst} <(m_r,m_s,m_t)$ for $1\leq l,s,t,r\leq n$. %Where $a_{ij}$ and $a_{rst}$ are ordered in the lexicographic order of their indices. 
Let $g_i$ be a generator of $Z_{m_i}$, $1 \leq i \leq n$. For each $\underline{a} \in \mathscr{A}$, define
\begin{equation}\label{2.1}
	\begin{aligned}
		& \omega_{\underline{a}}: G\times G \times G\rightarrow \mathbb{C}^{\times}\\
		&{\left[g_1^{i_1} \cdots g_n^{i_n}, g_1^{j_1} \cdots g_n^{j_n}, g_1^{k_1} \cdots g_n^{k_n}\right]} \\
		& \quad \mapsto \prod_{l=1}^n \zeta_{m_l}^{a_l i_l\left[\frac{j_l+k_l}{m_l}\right]} \prod_{1 \leq s<t \leq n} \zeta_{m_s}^{a_{s t} k_s\left[\frac{i_t+j_t}{m_t}\right]} \prod_{1 \leq r<s<t \leq n} \zeta_{\left(m_r, m_s, m_t\right)}^{a_{rst}k_rj_si_t}.
	\end{aligned}
\end{equation}

Here $\zeta_m$ represents an $m$-th primitive root of unity. The following gives us the desired expression.
\begin{lemma}[\cite{QQG} Proposition 3.8]
The set $\left\lbrace \omega_{\underline{a}}| \underline{a}\in A \right\rbrace $ forms a complete set of representatives of the normalized $3$-cocycles on $G$ up to $3$-cohomology.
\end{lemma}
\begin{rmk}
	\emph{We choose a slightly different representatives of normalized $3$-cocycles on $G$, as they are actually cohomologous to the formula (3.10) in \cite{QQG} for a fixed $\underline{a}$. We choose these representatives for convenience later.}
\end{rmk}
\subsection{Twisted quantum doubles}
 We recall the definition of the twisted quantum double for the completeness of the article.
\begin{definition}
	The twisted quantum double $D^{\omega}(G)$ of a finite group $G$ with respect to the $3$-cocycle $\omega$ on $G$ is the semisimple quasi-Hopf algebra with underlying vector space $(kG)^* \otimes kG$ in which multiplication, comultiplication $\Delta$, associator $\phi$, counit $\varepsilon$, antipode $\mathcal{S}$, $\alpha$ and $\beta$ are given by
	
	$$
	\begin{aligned}
		&(e(g) \otimes x)(e(h) \otimes y)=\theta_g(x, y) \delta_{g^x, h} e(g) \otimes x y, \\
		&\Delta(e(g) \otimes x)  =\sum_{h k=g} \gamma_x(h, k) e(h) \otimes x \otimes e(k) \otimes x, \\
		&\phi  =\sum_{g, h, k \in G} \omega(g, h, k)^{-1} e(g) \otimes 1 \otimes e(h) \otimes 1 \otimes e(k) \otimes 1, \\
		&\mathcal{S}(e(g) \otimes x)  =\theta_{g^{-1}}(x, x^{-1})^{-1} \gamma_x(g, g^{-1})^{-1} e(x^{-1} g^{-1} x) \otimes x^{-1}, \\
		&\varepsilon(e(g) \otimes x)  =\delta_{g, 1}, \quad \alpha=1, \quad \beta=\sum_{g \in G} \omega(g, g^{-1}, g) e(g) \otimes 1,
	\end{aligned}
	$$
where $\left\lbrace e(g)| g\in G\right\rbrace $ is the dual basis of $\left\lbrace g \in G\right\rbrace $,$\delta_{g,1}$ is the Kronecker delta, $g^x=x^{-1}gx$, and
\begin{align*}
	\theta_g(x,y)&=\frac{\omega(g,x,y)\omega(x,y,(xy)^{-1}gxy)}{\omega(x,x^{-1}gx,y)},\\
	\gamma_g(x,y)&= \frac{\omega(x,y,g)\omega(g,g^{-1}xg,g^{-1}yg)}{\omega(x,g,g^{-1}yg)}
\end{align*}
for any $x,y,g \in G$.\end{definition}
We may use $\DG$ to define abelian cocycles, which has been studied deeply in \cite{Abcocycle}.
\begin{definition}
	A $3$-cocycle $\omega$ on an abelian group $G$ is called abelian if $\DG$ is a commutative algebra.
\end{definition}
Using formula \eqref{2.1}, there's a nice description when the $3$-cocycle $\omega_{\underline{a}}$ is abelian:
\begin{lemma}[\cite{QQG}, Proposition 3.14] \label{lem2.5}
	The $3$-cocycle $\omega_{\underline{a}}$ is abelian if and only if 
	$$a_{rst}=0$$
	for all $1\leq r<s<t\leq n$.
\end{lemma}
\subsection{Module category and categorical Morita equivalence}
Module category is an important tool in the theory of tensor category. It is parallel to the module theory over a ring. The definition is similar to the definition of a tensor category.See \cite{tensor} Section 7 for explicit definitions.
 The theory of categorical Morita equivalence is a categorical analogue of Morita equivalence in ring theory, which plays an important role in the theory of module category.
 %Before we give the definition of categorical Morita equivalence, we introduce some basic definitions at first.
\begin{definition}
	Let $\MC$ be a tensor category with enough projective objects. A module category $\MM$ over $\MC$ is called exact if for any projective
	object $P \in \MC$ and any object $M \in \MM$ the object $P \otimes M$ is projective in $\MM$.
\end{definition}
For an exact indecomposable right module category, one can form the dual category $\mathscr{C}^*_{\mathscr{M}} := \operatorname{Fun}_{\MC}(\MM , \MM)$, that is, the category of module functors from $\MM$ to itself. It is known that
$\mathscr{C}^*_{\mathscr{M}}$ is also a tensor category.
 %by composition of functors, namely for $(\gamma^1,F_1)$, $(\gamma^2,F_2) \in \operatorname{Obj}(\mathscr{C}^*_{\mathscr{M}})$. Where $\gamma^1$ and $\gamma^2$ represent the module structures on the functors $F_1$and $F_2$ respectively, we define the tensor structure by 
%$$(\gamma^1,F_1) \otimes (\gamma^2,F_2):= (\gamma, F_1\circ F_2)$$
%where the module structure $\gamma$ is defined by $\gamma_{M,X}:= \gamma^1_{F_2(M),X} \circ F_1(\gamma^2_{M,X})$ for $M \in \MM$ and $X \in \MC$. For two morphisms $\eta: (\gamma^1,F_1) \rightarrow (\gamma^2, F_2)$ and $\eta': (\gamma'^{1},F_1') \rightarrow (\gamma'^{2}, F_2')$, their tensor product is $(\eta \otimes \eta')(M):=\eta_{F_2'(M)}\circ F_1(\eta'_M)$.

\begin{definition}
	Let $\MC$, $\mathscr{D}$ be tensor categories. We will say that $\MC$ and  $\mathscr{D}$ are categorical Morita equivalent if there is an exact indecomposable $\MC$-module category $\MM$ and a tensor equivalence $\mathscr{D}^{op}\cong \mathscr{C}^*_{\mathscr{M}}$.
\end{definition}
Here is a basic example of categorical Morita equivalence.
\begin{exm}\rm
	Let $G$ be a ﬁnite group and let $\MC= \operatorname{Vec}_G$. The category $\operatorname{Vec}$ is an exact $\operatorname{Vec}_G$-module category via the forgetful tensor functor $\operatorname{Vec}_G \rightarrow\operatorname{Vec}$. Consider the dual category 
	$(\operatorname{Vec}_G)^*_{\operatorname{Vec}}$. By deﬁnition, a $ \operatorname{Vec}_G$-module endofunctor $F$ of \ $\operatorname{Vec}$ consists of a vector space $V:= F(\vmathbb{k})$ and a collection of isomorphisms
	$$\gamma_g \in \operatorname{Hom}(F(\delta_g\otimes \vmathbb{k}),\delta_g \otimes F(\vmathbb{k})))=\operatorname{End}_{\vmathbb{k}}(V).$$
	By axiom of module functor, the map $g \mapsto \gamma_g: G \longrightarrow \operatorname{GL}(V)$ is a representation of $G$ on $V$. Conversely, any such representation determines a $\operatorname{Vec}_G$-module endofunctor of  $\operatorname{Vec}$. The homomorphisms of representations are precisely morphisms between the corresponding module functors. Thus, $(\operatorname{Vec}_G)^*_{\operatorname{Vec}}\cong  \operatorname{Rep}(G)^{\operatorname{op}}$, i.e., the categories $\operatorname{Vec}_G$  and $\operatorname{Rep}(G)$ are categorical Morita equivalent.
\end{exm}

\section{On  gauge equivalence between $D^{\omega}(G)$ and $D(G')$}
Throughout this section, let $G=Z_{m_1}\times Z_{m_2}\times \cdots \times Z_{m_n}$ with $m_i \mid m_{i+1}$ for $1\leq i \leq n-1$ and  $\omega$ be a normalized $3$-cocycle with the following form:
\begin{align}\label{3.05}
	& \omega\left( g_1^{i_1} \cdots g_n^{i_n}, g_1^{j_1} \cdots g_n^{j_n}, g_1^{k_1} \cdots g_n^{k_n}\right)  = \prod_{l=1}^n \zeta_{m_l}^{a_l i_l\left[\frac{j_l+k_l}{m_l}\right]} \prod_{1 \leq s<t \leq n} \zeta_{m_s}^{a_{s t} k_s\left[\frac{i_t+j_t}{m_t}\right]} \prod_{1 \leq r<s<t \leq n} \zeta_{\left(m_r, m_s, m_t\right)}^{a_{r s t} k_r j_s i_t}.
\end{align}\par
\subsection{Categorical Morita equivalence of pointed fusion categories}
We first recall the result of categorical Morita equivalence in \cite{Morita}:
\begin{lemma}[\cite{Morita} Theorem 3.9] \label{Lem3.1}
	Let $G$ and $\hat{G}$ be ﬁnite groups, $\eta \in  Z^3(G, \CC)$ and $\hat{\eta} \in  Z^3(\hat{G}, \CC)$ be normalized $3$-cocycles. Then the tensor categories $\text{Vec}_G^{\eta}$ and $\text{Vec}_{\hat{G}}^{\hat{\eta}}$ are categorical Morita equivalent if and only if the following conditions are satisfied:\par
$\operatorname{(1)}$ There exist isomorphism of groups:
	\begin{equation}\label{3.1}
			\phi: H \underset{F}{\rtimes}K \stackrel{\cong}{\longrightarrow}G, \quad \hat{\phi}: =\widehat{H} \underset{\hat{F}}{\rtimes} K \stackrel{\cong}{\longrightarrow} \hat{G}
	\end{equation}
	for some ﬁnite group $K$ acting on the abelian normal group $H$, with $F \in Z^2 (K, H)$ and $\hat{F}\in Z^2 (K, \widehat{H})$ where $\widehat{H}:= \operatorname{Hom}(H, \CC)$.\par
	$\operatorname{(2)}$  There exists $\varepsilon: K^3 \longrightarrow \CC$ such that 
	\begin{equation}\label{3.2}
		\hat{F} \wedge F= \delta_K \varepsilon.	\end{equation}
	Here $\hat{F} \wedge F(k_1,k_2,k_3,k_4):=\hat{F}(k_1,k_2)\left( F(k_3,k_4)\right). $\par 
$\operatorname{(3)}$ The cohomology classes satisfy the equations $[ \phi^* \eta ] = [ \omega ]$ and $[ \hat{\phi}^*\hat{\eta} ] = [ \hat{\omega} ]$ with

	\begin{equation}\label{3.3}
		\begin{aligned}
			& \omega\left(\left(h_1, k_1\right),\left(h_2, k_2\right),\left(h_3, k_3\right)\right):=\hat{F}\left(k_1, k_2\right)\left(h_3\right) \varepsilon\left(k_1, k_2, k_3\right), \\
		& \hat{\omega}\left(\left(\rho_1, k_1\right),\left(\rho_2, k_2\right),\left(\rho_3, k_3\right)\right):=\varepsilon\left(k_1, k_2, k_3\right) \rho_1\left(F\left(k_2, k_3\right)\right).
		\end{aligned}
	\end{equation}

\end{lemma}
For simplicity, we will regard $\omega$ (resp. $\omega'$) as a normalized $3$-cocycle on $G$ and $H\underset{\hat{F}}{\rtimes} K$ (resp. $G'$ and $\widehat{H} \underset{\hat{F}}{\rtimes} K$ ) simultaneously in the following context.
A simple but useful application of this lemma is given as follows:
\begin{cor} \label{cor3.2}
Let $G$ be a finite abelian group. If $\operatorname{Vec}^{\omega}_G$ is categorical Morita equivalent to $\operatorname{Vec}_{G'}$ for a finite group $G'$, then
\begin{itemize}
    \item[\emph{(i)}]  The choice of $\varepsilon$ in Lemma \ref{Lem3.1} must be $\varepsilon(k_1,k_2,k_3) =1$ for all $k_1$, $k_2$, $k_3 \in K$.
 \item[\emph{(ii)}] The crossed product $G= H \underset{F}{\rtimes} K$ in Lemma $\ref{Lem3.1}$ is actually a direct product. That is, the decomposition of $G$ must be of the form $G= H\times K$ for an abelian normal subgroup $H$.
\end{itemize}
	\end{cor}
\begin{proof}
	 Suppose $\operatorname{Vec}^{\omega}_G$ is categorical Morita equivalence to $\operatorname{Vec}_{G'}$. By Lemma \ref{Lem3.1}, There exists isomorphism of groups:
	 $H \underset{F}{\rtimes}K \stackrel{\cong}{\longrightarrow}G, \ \widehat{H} \underset{\hat{F}}{\rtimes} K \stackrel{\cong}{\longrightarrow} \hat{G}
	 $ for abelian normal subgroup $H$ of $G$. By assumption, the normalized $3$-cocycle $\omega'$ of $G'$ is trivial. That is $$ {\omega'}\left(\left( \rho_1,k_1\right),\left(\rho_2,k_2\right),\left( \rho_3,k_3\right)\right)=\varepsilon\left(k_1, k_2, k_3\right) \rho_1\left(F\left(k_2, k_3\right)\right)\equiv 1$$for all $k_1, k_2, k_3 \in K$ and $\rho_1, \rho_2, \rho_3 \in \widehat{H}$.\par 
	We first assume $\varepsilon$ is nontrivial, then there will exist $k', k'', k''' \in K$ such that $\varepsilon(k',k'',k''') \neq 1$, then
	$${\omega'\left( \left(  1_{\widehat{H}},k'\right) ,\left( \rho_2,k''\right) ,\left( \rho_3,k''',\right) \right) }=\varepsilon(k',k'',k''')1_{\widehat{H}}\left( F(k_1,k_2)\right)=\varepsilon(k',k'',k''')\neq 1.$$
	This implies $\omega'$ will never be identically equal to $1$, which is a contradiction.
	\par Suppose the crossed product is not a direct product. Then $F \in Z^2(K,H)$ is nontrivial, and there will exist $k'$, $k'' \in K$ such that $F(k',k'') \neq 1_H$. So we can choose a character $\rho \in \hat{H}$ such that $\rho\left( F(k',k'')\right)  \neq 1$ and  consider the ratio of
$$\frac{\omega'\left( \left( \rho,k_1\right) ,\left( 1_{\hat{H}},k_2\right) ,\left( 1_{\hat{H}},k_3\right) \right) }{\omega'\left( \left(  1_{\hat{H}},k_1\right) ,\left( 1_{\hat{H}},k_2\right),\left( 1_{\hat{H}},k_3\right)\right) }=\rho\left( F(k_1,k_2)\right) \neq 1.$$
Thus one of the values of $\omega'$ can't be one. This  leads to a contradiction as well.
\end{proof}
\subsection{The first main result}
Keep the notation above, we will give a sufficient  condition of categorical Morita equivalence between $\operatorname{Vec}^{\omega}_G$ and $\operatorname{Vec}_{G'}$ in this subsection. Now let $G=Z_{m_1}\times Z_{m_2}\times \cdots \times Z_{m_n}$ with $m_i \mid m_{i+1}$ for $1\leq i \leq n-1$. The $3$-cocycle $\omega $ as in (\ref{3.05}). $$\underline{a}=(a_1,a_2,...,a_l,...,a_n,a_{12},a_{13},...,a_{st},...,a_{n-1,n},a_{123},...,a_{rst},...a_{n-2,n-1,n}) \in \mathscr{A}$$ where $0 \leq a_l <m_l$, $0\leq a_{st} <(m_s,m_t)$, $0\leq a_{rst} <(m_r,m_s,m_t)$. For a fixed $\underline{a} \in \mathscr{A}$,
	define the following sets:
	\begin{align*}
		&A_1:=\left\lbrace i | a_{ij}\neq 0, 1\leq i <j \leq n \right\rbrace, \qquad A_2:=\left\lbrace i | a_{ijk}\neq 0, 1\leq i <j<k \leq n \right\rbrace ,\\
		&B_1:= \left\lbrace j | a_{ij}\neq 0, 1\leq i <j \leq n \right\rbrace, \qquad
		B_2:=\left\lbrace j,k | a_{ijk}\neq 0, 1\leq i <j<k \leq n \right\rbrace.
	\end{align*}
	Let $A=A_1 \cup A_2$, $B=B_1 \cup B_2$.
\begin{thm}\label{thm1.2}
	Let $G$ be a finite abelian group and $\omega$ is a normalized $3$-cocycle on $G$ as in (\ref{(1.3)}). If 
 \begin{itemize}
     \item[\emph{(i)}] $a_i=0$ for all $ 1\leq i \leq n$,
     \item[\emph{(ii)}] $A \cap B =\emptyset$.
 \end{itemize}
Then $\operatorname{Vec}^{\omega}_G$ is categorical Morita equivalent to $\operatorname{Vec}_{G'}$ for a finite group $G'.$ 
\end{thm}
\begin{proof}
Let $a_i=0$ for all $ 1\leq i \leq n$ and $A \cap B =\emptyset$. Denote $I=\left\lbrace 1,2,\cdots,n\right\rbrace $. Clearly $A,B \subset I$. Now take $H=\underset{i\in A}{\prod} Z_{m_i}=\underset{i\in A}{\prod} \langle g_i \rangle$, and  $K=\underset{j \in I \setminus A}{\prod}Z_{m_j}=\underset{j \in I \setminus A}{\prod}\langle g_j \rangle$, then $G \cong H \times K$. Define 
$$\hat{F}\left( \underset{m \in I \setminus A }{\prod}g_m^{i_m},\underset{m \in I \setminus A }{\prod}g_m^{j_m}\right) = \underset{p\in A_1, q \in B_1}{\prod_{p<q
}}(\chi_p)^{a_{pq}[\frac{i_q+j_q}{m_q}]}\underset{r \in A_2, s,t \in B_2}{\prod_{r<s<t}}(\chi_r)^{\frac{a_{rst}m_r}{(m_r,m_s,m_t)}j_si_t}$$
where $\chi_p \in \widehat{Z_{m_p}}$ is primitive. $\hat{F}$  lies in $Z^2(K,\hat{H})$ by direct computation.
We are now going to show  $\operatorname{Vec}^{\omega}_G$ is categorical Morita equivalent to $\operatorname{Vec}_{H\underset{\hat{F}}{\rtimes}K}$ by Lemma \ref{Lem3.1}: \par
Equation (\ref{3.1}) has been done. If we set $\varepsilon: K^3\rightarrow k^{\times}$ being identical to 1, then (\ref{3.2}) is satisfied since $\hat{F} \wedge F(k_1,k_2,k_3,k_4)=\hat{F}(k_1,k_2)\left( F(k_3,k_4)\right) =1=\delta_K\varepsilon$. Note
\begin{align*}
	\omega\left( g_1^{i_1} \cdots g_n^{i_n}, g_1^{j_1} \cdots g_n^{j_n}, g_1^{k_1} \cdots g_n^{k_n}\right)  &= \prod_{1 \leq p<q \leq n} \zeta_{m_p}^{a_{p q} k_p\left[\frac{i_q+j_q}{m_q}\right]} \prod_{1 \leq r<s<t \leq n} \zeta_{\left(m_r, m_s, m_t\right)}^{a_{r s t} k_r j_s i_t}\\
	&=\underset{p\in A_1, q \in B_1}{\prod_{p<q
	}}\zeta_{m_p}^{a_{pq}k_p[\frac{i_q+j_q}{m_q}]}\underset{r \in A_2, s,t \in B_2}{\prod_{r<s<t}}\zeta_{(m_r,m_s,m_t)}^{a_{rst}k_rj_si_t}
\end{align*}
since $a_{pq}=0$ if $p \notin A_1$ or $q \notin B_1$  and $a_{rst}=0$ if $r \notin A_2$ or $s \notin B_2$ or $t \notin B_2$.
On the other hand 
\begin{align*}
	\hat{F}\left( \underset{n \in I \setminus A }{\prod}g_n^{i_n},\underset{n \in I \setminus A }{\prod}g_n^{j_n}\right) (\prod_{m \in A}g_m^{k_m})&=\underset{p\in A_1, q \in B_1}{\prod_{p<q
	}}\zeta_{m_p}^{a_{pq}k_p[\frac{i_q+j_q}{m_q}]}\underset{r \in A_2, s,t \in B_2}{\prod_{r<s<t}}\zeta_{m_r}^{a_{rst}\frac{m_rk_rj_si_t}{(m_r,m_s,m_t)}}\\
&=\underset{p\in A_1, q \in B_1}{\prod_{p<q
}}\zeta_{m_p}^{a_{pq}k_p[\frac{i_q+j_q}{m_q}]}\underset{r \in A_2, s,t \in B_2}{\prod_{r<s<t}}\zeta_{(m_r,m_s,m_t)}^{a_{rst}k_rj_si_t}.
\end{align*}
Hence \begin{align*}
	\omega\left( g_1^{i_1} \cdots g_n^{i_n}, g_1^{j_1} \cdots g_n^{j_n}, g_1^{k_1} \cdots g_n^{k_n}\right)&=\omega( (\prod_{m\in A}g_m^{i_m},\prod_{n\in I \setminus A} g_n^{i_n}),(\prod_{m\in A}g_m^{j_m},\prod_{n\in I \setminus A} g_n^{j_n}),(\prod_{m\in A}g_m^{k_m},\prod_{n\in I \setminus A} g_n^{k_n})) \\
	&=\underset{p\in A_1, q \in B_1}{\prod_{p<q
	}}\zeta_{m_p}^{a_{pq}k_p[\frac{i_q+j_q}{m_q}]}\underset{r \in A_2, s,t \in B_2}{\prod_{r<s<t}}\zeta_{(m_r,m_s,m_t)}^{a_{rst}k_rj_si_t}\\
&=\hat{F}(\prod_{n\in I \setminus A}g_n^{i_n},\prod_{n\in I \setminus A} g_n^{j_n})(\prod_{m\in A}g_m^{k_m}).
\end{align*}
Thus the first equation of (\ref{3.3}) has been verified.\par 
Since $G \cong H \times K = H \underset{F}{\rtimes} K$ where $F(k_1,k_2)=1_H$ for all $k_1,k_2 \in  K$. Then $\rho(F(k_1,k_2))=1$ for all $\rho \in \hat{H}$ and $k_1, k_2 \in K$. Thus $$
{\omega'}\left(\left( \rho_1,k_1\right),\left(\rho_2,k_2\right),\left( \rho_3,k_3\right)\right)=1=\rho_1(F(k_2,k_3)).$$
We have verified all conditions in Lemma \ref{Lem3.1}. Hence $\operatorname{Vec}^{\omega}_G$ is categorical Morita equivalent to $\operatorname{Vec}_{\hat{H}\underset{\hat{F}}{\rtimes}K}$ if $a_i=0$ for all $ 1\leq i \leq n$ and $A \cap B =\emptyset$.
\end{proof}
%\begin{lemma}[\cite{center}]\label{lem3.7}
	%Let $G$ be a finite group and $\omega$ is a $3$-cocycle on $G$, then $\mathcal{Z}(\operatorname{Vec}_G^{\omega})$ is braided fusion equivalent to $\operatorname{Rep}(D^{\omega}(G))$ 
%\end{lemma}
%\begin{lemma}\label{lem3.8}
	%Two fusion categories are categorical Morita equivalent if and only if their centers are braided fusion equivalent.
%\end{lemma}
%\begin{lemma}[\cite{gauge} Theorem2.2]\label{lem3.9}
	%Let $H$ and $H'$ be two finite-dimensional quasi-Hopf algebras.Then $\operatorname{Rep}(H)$ is tensor equivalent to $\operatorname{Rep}(H')$ if and only if $H$ is gauge equivalent to $H'$.
%\end{lemma}
This theorem implies Theorem \ref{thm1.1} directly.
\begin{proof}[Proof of Theorem \ref{thm1.1}]
According to Theorem \ref{thm1.2}, if $a_i = 0$ for all $1\leq i \leq n$ and $A \cap B = \emptyset$, then $\operatorname{Vec}_G^{\omega}$ is categorical Morita equivalent to $\operatorname{Vec}_{G'}$ for some finite group $G'$. By Theorem 3.1 in \cite{tensor}, the centers of these two fusion categories are braided equivalent. It is known that the center is equivalent to the representation category of the corresponding Drinfeld double (see for example \cite{center}). That is, $\operatorname{Rep}(D^{\omega}(G))$ is braided tensor equivalent to $\operatorname{Rep}(D(G'))$. Hence 
 $D^{\omega}(G)$ will be  gauge equivalent to $D(G')$ by Theorem 2.2 in \cite{gauge}.
\end{proof}
A natural question is when $G'$ can be a finite abelian group in the theorem above. Here is the answer.
\begin{cor}
	\label{cor3.6}
		Let $G$ be a finite abelian group and $\omega$  a normalized $3$-cocycle on $G$ as above.  Then $\operatorname{Vec}^{\omega}_G$ is categorical Morita equivalent to $\operatorname{Vec}_{G'}$ for a finite abelian group $G'$ if  \par 
$\operatorname{(1)}$ $a_i=0$ for all $ 1\leq i \leq n$ and  $a_{rst}=0$ for all $1 \leq r < s < t\leq n$,\par 
$\operatorname{(2)}$  $A_1 \cap B_1 =\emptyset$.
\end{cor}
\begin{proof}
	We first assume the condition (i) and (ii) in Theorem \ref{thm1.2} hold. Thus there's a finite group $G'$ such that $\operatorname{Vec}^{\omega}_G$ is categorical Morita equivalence to $\operatorname{Vec}_{G'}$. By construction, $ G' \cong (\underset{i\in A}{\prod} Z_{m_i} ) \underset{\hat{F}}{\rtimes}(\underset{j \in I \setminus A}{\prod}Z_{m_j})$ where $\hat{F}$ is defined to be
	$$\hat{F}\left( \underset{m \in I \setminus A }{\prod}g_m^{i_m},\underset{m \in I \setminus A }{\prod}g_m^{j_m}\right) = \underset{p\in A_1, q \in B_1}{\prod_{p<q
	}}(\chi_p)^{a_{pq}[\frac{i_q+j_q}{m_q}]}\underset{r \in A_2, s,t \in B_2}{\prod_{r<s<t}}(\chi_r)^{\frac{a_{rst}m_r}{(m_r,m_s,m_t)}j_si_t}.$$

\begin{align*}
\text{Then} \ G' \ &\text{is abelian} 	 \Leftrightarrow  (\underset{m \in A}{\prod}g_m^{i_m},\underset{n \in I \setminus A}{\prod}g_n^{i_n})\cdot  (\underset{m \in A}{\prod}g_m^{j_m},\underset{n \in I \setminus A}{\prod}g_n^{j_n})=  (\underset{m \in A}{\prod}g_m^{j_m},\underset{n \in I \setminus A}{\prod}g_n^{j_n}) \cdot (\underset{m \in A}{\prod}g_m^{i_m},\underset{n \in I \setminus A}{\prod}g_n^{i_n})\\
&\Leftrightarrow  \hat{F}\left( \underset{n \in I \setminus A }{\prod}g_n^{i_n},\underset{n \in I \setminus A }{\prod}g_n^{j_n}\right)=\hat{F}\left( \underset{n \in I \setminus A }{\prod}g_n^{j_n},\underset{n \in I \setminus A }{\prod}g_n^{i_n}\right) \\
&\Leftrightarrow \underset{p\in A_1, q \in B_1}{\prod_{p<q
}}(\chi_p)^{a_{pq}[\frac{i_q+j_q}{m_q}]}\underset{r \in A_2, s,t \in B_2}{\prod_{r<s<t}}(\chi_r)^{\frac{a_{rst}m_rj_si_t}{(m_r,m_s,m_t)}}=\underset{p\in A_1, q \in B_1}{\prod_{p<q
}}(\chi_p)^{a_{pq}[\frac{j_q+i_q}{m_q}]}\underset{r \in A_2, s,t \in B_2}{\prod_{r<s<t}}(\chi_r)^{\frac{a_{rst}m_ri_sj_t}{(m_r,m_s,m_t)}}\\
&\Leftrightarrow A_2, B_2 =\emptyset. 
\end{align*} 
This is equivalent to $a_{rst}=0$ for all $1 \leq r < s< t\leq n$. Thus if (1) and (2) hold, then $G'$ is abelian.
\end{proof}
If $G$ is a cyclic group, then conditions (i),(ii) in Theorem \ref{thm1.2} are also necessary. In fact, for any cyclic group $G \cong Z_m=<g|g^m=1>$ with a normalized $3$-cocycle $\omega_a$ given by $\omega_a(g^i,g^j,g^k)=\zeta_m^{ai[\frac{j+k}{m}]}$, where $0 \leq a,i,j,k<m$, we have (noting that the condition (ii) is always satisfied now)
\begin{prop} \label{Prop3.5}
The fusion category $\operatorname{Vec}_{G}^{\omega_a}$  is categorical Morita equivalent to $\operatorname{Vec}_{G'}$ for a finite group $G'$ if and only if $a=0$.
\end{prop}
\begin{proof}
	The sufficiency follows from Theorem \ref{thm1.2}. Now 
	suppose that $\operatorname{Vec}_{G}^{\omega_{a}}$  is categorical Morita equivalent to $\operatorname{Vec}_{G'}$ for a finite group $G'$. By Corollary \ref{cor3.2}, $G$ must be direct product of two subgroups, like $G\cong H\times K$ and the function $\varepsilon$ should be $1$. Since $G$ is cyclic, then $H$ and  $K$ must be  cyclic subgroups. Moreover, $|H|$ should be prime to $|K|$, hence $H^2(K,\hat{H})=\left\lbrace 1\right\rbrace$. Thus $\omega_a$ should be $1$ by formula (\ref{3.3}). That is, $a=0.$
\end{proof}
But in general the conditions (i) and (ii) in Theorem \ref{thm1.2} both are not necessary as the following example shows.
\begin{exm}\rm
	Let $G \cong Z_2 \times Z_2\times Z_2= \langle g_1 \rangle \times \langle g_2 \rangle \times \left\langle g_3\right\rangle $ and 
$$\omega\left( g_1^{i_1}  g_2^{i_2}g_3^{i_3}, g_1^{j_1}  g_2^{j_2}g_3^{j_3}, g_1^{k_1}  g_2^{k_2}g_3^{k_3}\right)  =(-1)^{i_2[\frac{j_2+k_2}{2}]}(-1)^{k_1[\frac{i_2+j_2}{2}]}(-1)^{k_1[\frac{i_3+j_3}{2}]}(-1)^{k_2[\frac{i_3+j_3}{2}]}.$$
In this case, $\underline{a}=(0,1,0,1,1,1,0)\in \mathscr{A}$, $a_2 \neq 0$ and  $A_1 \cap B_1 \neq 0$.\par 
Take $H \cong Z_2=\langle g_1 \rangle$ and $K \cong Z_2\times Z_2=\langle g_1g_2\rangle \times \left\langle g_3\right\rangle $. Obviously, $G \cong H \times K$. Define 
$$\hat{F}: K \times K \rightarrow \hat{H}, \  \ \ \hat{F}((g_1g_2)^{i_2},g_3^{i_3},(g_1g_2)^{j_2},g_3^{j_3})= \chi_1^{[\frac{i_2+j_2}{2}]}\chi_1^{[\frac{i_3+j_3}{2}]},$$
where $\chi_1$ generates $\hat{H}$. Let $G'=\hat{H} \underset{\hat{F}}{\rtimes} K$, we are going to show $\VG$ is categorical Morita equivalent to $\operatorname{Vec}_{G'}$.\par 
Define $\varepsilon: K^3 \rightarrow \mathbb{C}^*$ as $\varepsilon \equiv 1$, then equation (\ref{3.2}) holds. 
Note that 
\begin{align*}
	\omega(g_1^{(i_1-i_2)'}&((g_1g_2)^{i_2},g_3^{i_3}),g_1^{(j_1-j_2)'}((g_1g_2)^{j_2}g_3^{j_3}),g_1^{(k_1-k_2)'}((g_1g_2)^{k_2}g_3^{k_3}))\\&=\omega\left( g_1^{i_1}  g_2^{i_2}g_3^{i_3}, g_1^{j_1}  g_2^{j_2}g_3^{j_3}, g_1^{k_1}  g_2^{k_2}g_3^{k_3}\right)	=(-1)^{i_2[\frac{j_2+k_2}{2}]}(-1)^{k_1[\frac{i_2+j_2}{2}]}(-1)^{k_1[\frac{i_3+j_3}{2}]}(-1)^{k_2[\frac{i_3+j_3}{2}]},
\end{align*}
 and 
$$\widehat{F}\left( ((g_1g_2)^{i_2},g_3^{i_3}),((g_1g_2)^{j_2},g_3^{j_3})\right) (g_1^{(k_1-k_2)'}) =(-1)^{(k_1-k_2)'[\frac{i_2+j_2}{2}]}(-1)^{(k_1-k_2)'[\frac{i_3+j_3}{2}]}$$
for $0 \leq i_1,i_2,j_1,j_2,k_1,k_2 \leq 1$. Actually, 
\begin{align*}
	&\frac{(-1)^{i_2[\frac{j_2+k_2}{2}]}(-1)^{k_1[\frac{i_2+j_2}{2}]}(-1)^{k_1[\frac{i_3+j_3}{2}]}(-1)^{k_2[\frac{i_3+j_3}{2}]}}{(-1)^{(k_1-k_2)'[\frac{i_2+j_2}{2}]}(-1)^{(k_1-k_2)'[\frac{i_3+j_3}{2}]}} \\
	&=\frac{(-1)^{i_2[\frac{j_2+k_2}{2}]}(-1)^{k_1[\frac{i_2+j_2}{2}]}(-1)^{k_1[\frac{i_3+j_3}{2}]}(-1)^{k_2[\frac{i_3+j_3}{2}]}}{(-1)^{k_1[\frac{i_2+j_2}{2}]}{(-1)^{-k_2[\frac{i_2+j_2}{2}]}(-1)^{k_1[\frac{i_3+j_3}{2}]}(-1)^{-k_2[\frac{i_3+j_3}{2}]}} }\\
	&= (-1)^{k_2[\frac{i_2+j_2}{2}]}\cdot(-1)^{i_2[\frac{j_2+k_2}{2}]} =1.
\end{align*}
Thus $\omega\left( g_1^{i_1}  g_2^{i_2}g_3^{i_3}, g_1^{j_1}  g_2^{j_2}g_3^{j_3}, g_1^{k_1}  g_2^{k_2}g_3^{k_3}\right)=\widehat{F}\left( ((g_1g_2)^{i_2},g_3^{i_3}),((g_1g_2)^{j_2},g_3^{j_3})\right) (g_1^{(k_1-k_2)'})$ and the first equation in (\ref{3.3}) holds. Obviously, if we define 
the $3$-cocycle $\omega'$ on $G'$ as 
$$\omega'\equiv1.$$
Then the second equation in (\ref{3.3}) holds. Hence we have proved $\VG$ is categorical Morita equivalent to $\operatorname{Vec}_{G'}$.
\end{exm}
%\subsection{ Twisted quantum double versus ordinary quantum double }
%In this section, we will give a sufficient  condition  when $D^{\omega}(G)$  will never be gauge equivalent to $D(G')$ for a finite group $G'$.
%\begin{definition}
	%Let $G$ be a finite abelian group and $\omega$ is a $3$-cocycle on %$G$. We say $D^{\omega}(G)$ is genuine if it will never be gauge %equivalent to $D(G')$ for a finite group $G'$.
%\end{definition}
%Recall that the tensor category $\operatorname{Rep}(D^{\omega}(G))$ has a deep relation between the center of $\operatorname{Vec}_G^{\omega}$:

%Now we state the proof of Theorem \ref{thm1.1}
%\begin{thm} \label{thm3.10}
%	Let $G=Z_{m_1}\times Z_{m_2}\times \cdots Z_{m_n}$ and $\omega$ is a $3$-cocycle on $G$ as in (\ref{2.1}), then $D^{\omega}(G)$ is genuine if and only if one of the following condition holds:  \par 
%		(i) There exist some $i$ such that $a_i \neq 0$.\par 
%	(ii) $A \cap B \neq \emptyset$. 
%\end{thm}

%The next Corollary follows directly from Proposition \ref{Prop3.5} 
%\begin{cor}\label{Cor3.10}
   % Let $G$ be a finite cyclic group with a nontrivial $3$-cocycle $\omega$, then $D^{\omega}(G)$  will never be gauge equivalent to $D(G')$ for a finite group $G'$.
%\end{cor}
\section{On genuineness of twisted quantum double }
In the article \cite{DPR}, the authors asked whether $\DG$ can be obtained by twisting a Hopf algebra or not. 
In \cite{Abcocycle} Example 9.5, the authors have shown that $D^{\omega}(Z_2)$ is genuine for $\omega$  being the normalized $3$-cocycle on $G$ whose cohomology class is nontrivial. This gives a negative answer to the above question. In this section, we will investigate when a twisted quantum double of a cyclic group to be genuine, that is, it can't be  obtained by twisting a Hopf algebra.
\subsection{The structure of $\DG$}
Let $G$ be an abelian group and $\omega$  an abelian $3$-cocycle on $G$. Let $\Gamma^{\omega}$ be the group of all group-like elements in $\DG$, and denote $ \omega_g(x,y)=\frac{\omega(g,x,y)\omega(x,y,g)}{\omega(x,g,y)}$ for $g,x,y \in G$. 
\begin{lemma}[\cite{Abcocycle} Corollary 3.6]
	With the notation above, $\DG$ is spanned by the set of group-like elements $\Gamma^{\omega}$ and it is a commutative algebra. In particular, $\omega_g$ is a $2$-coboundary for any $g \in G$.
\end{lemma}
Moreover, $\Gamma^{\omega}$ can be seen as an abelian extension, which may help us to figure out the explicit structure of  $\DG$.
\begin{lemma}[\cite{Abcocycle} Proposition 3.8]
	Let $\hat{G}$ be the character group of $G$, then	$\Gamma^{\omega}$  is an extension 
	\begin{equation}\label{(3.4)}
		1 \longrightarrow \hat{G} \longrightarrow \Gamma^{\omega} \longrightarrow G \longrightarrow 1.
	\end{equation}
	For each $g \in G$, let $\omega_g= \delta \tau_g$ for a $1$-cochain $\tau_g: G \rightarrow \mathbb{C}^{\times}$. The $2$-cocycle $\beta$ associated to this central extension is given by 
	\begin{equation}\label{(3.5)}
		\beta(x,y)(g)= \frac{\tau_x(g)\tau_y(g)}{\tau_{xy}(g)}\omega_g(x,y).
	\end{equation}
\end{lemma}

From  now on, let $G=Z_m=\langle g \rangle$ be a finite cyclic group and $\omega(g^i,g^j,g^k)= \zeta_m^{ai[\frac{j+k}{m}]}$ be  a nontrivial normalized $3$-cocycle. In this case, $\widehat{G}=\widehat{Z_m}=\langle \chi \rangle$, where $\chi(g)=\zeta_m$. We will determine when $\DG$ is genuine. 
The first  task is to figure out the group structure on $\GO$. Since $\GO$ is totally determined by $\DG$,   it is independent of  the choice of $\tau_x$ for each $ x\in G$. 
\begin{lemma}
	Let 
	$\tau_{g^i}(g^j)=\zeta_{m^2}^{aij}$ for all $0\leq i,j \leq m$. then $\delta \tau_{g^i}=\omega_{g^i}$. Further, $\beta(g^i, g^j)=\chi^{2a[\frac{i+j}{m}]}$ in this case.
\end{lemma}
\begin{proof}
	Direct computation shows that
	$$ \delta \tau_{g^i}(g^j,g^k)=\frac{\tau_{g^i}(g^j)\tau_{g^i}(g^k)}{\tau_{g^i}(g^{(j+k)'})}=\frac{\zeta_{m^2}^{aij}\zeta_{m^2}^{aik}}{\zeta_{m^2}^{ai(j+k)'}}=\zeta_m^{ai[\frac{j+k}{m}]}=\omega_{g^i}(g^j,g^k).
	$$
	Here $(j+k)'$ denotes the reminder of $j+k$ modulo $m$.
	Now 
	$$\beta(g^j,g^k)(g^i)=\frac{\tau_{g^j}(g^i)\tau_{g^k}(g^i)}{\tau_{g^{(j+k)}(g^i)}}\omega_{g^i}(g^j,g^k)=\zeta_m^{2ai[\frac{j+k}{m}]}.$$
	Hence $\beta(g^i, g^j)=\chi^{2a[\frac{i+j}{m}]}.$
\end{proof}
Note that $\Gamma^{\omega}$ consists of all group-like elements, hence it is benefit to write down explicit formulas of all group-like elements. By \cite{Abcocycle}, a nonzero element $u$ in $\DG$ is a group-like element if and only if \begin{equation} \label{(3.55)}
	u= \sigma_{\tau}(\alpha,x)= \sum_{g \in G}\alpha(g)\tau_x(g)e(g) \otimes x.
\end{equation}
for  $\alpha \in \widehat{G}$ and $x\in G$. Here we have assumed $G$ is a cyclic group, we can simplify the expression of $\sigma_{\tau}(\alpha, x)$. 
\begin{lemma}\begin{itemize}
    \item[\emph{(i)}] We have $\sigma_{\tau}(\chi^j, 1)=\chi^j \otimes 1$ and $\sigma_{\tau}(\chi^j,g)=\sum_{i=0}^{m-1}\zeta_{m^2}^{ai}\zeta_m^{ij} e(g^i)\otimes g$, where $ 0 \leq j \leq m-1$.
     \item[\emph{(ii)}] Let $s=\sigma_{\tau}(\chi, 1), t=\sigma_{\tau}(1,g)$, then $\Gamma^{\omega}$ has the following presentation: \begin{equation}\label{(3.65)}		\left\langle s, t| t^{\frac{m^2}{(m,2a)}}=s^m=1, \ s^{2a}=t^{m},\ st=ts \right\rangle.
	\end{equation}
\end{itemize}

\end{lemma}
\begin{proof}
	First, by direct computation \begin{align}\label{(3.66)}
	e(g^i)=1_i&=\frac{1}{m}\sum_{l=0}^{m-1}\zeta_m^{-li}\chi^l, \\
		\chi^j&=\sum_{i=0}^{m-1}\zeta_m^{ij}e(g^i).
	\end{align} 
Then  $$\sigma_{\tau}(\chi^j,1)=\sum_{i=0}^{m-1}\chi^j(g^i)\tau_1(g^i)e(g^i)\otimes 1= \sum_{i=0}^{m-1}\zeta_m^{ij}e(g^i)\otimes 1=\chi^j \otimes 1,$$
and 
$$\sigma_{\tau}(\chi^j,g)=\sum_{i=0}^{m-1}\chi^j(g^i)\tau_{g}(g^i)e(g^i)\otimes g=\sum_{i=0}^{m-1}\zeta_{m^2}^{ai}\zeta_m^{ij}e(g^i) \otimes g.$$
By multiplication rule of twisted quantum double, 
$$\stg \cdot \stx=\sigma_{\tau}(\chi,g)=\stx \cdot \stg.$$
Suppose $0 \leq l \leq m-1$, we have
$$\stg^l=\sum_{i=0}^{m-1}\zeta_{m^2}^{ail}e(g^i)\otimes g^l=\sigma_{\tau}(1,g^l).$$ Moreover, 
 $$\stg^m=\sum_{i=0}^{m-1}\zeta_{m^2}^{mai}e(g^i)\theta_{g^i}(g^{m-1},g)\otimes 1=\sum_{i=0}^{m-1}\zeta_m^{2ai}e(g^i)\otimes 1=\chi^{2a}\otimes 1=\stx^{2a}.$$
It is easy to verify that $\stx^{m}=1$ and thus $\stx^{2a}$ has order $\frac{m}{(m,2a)}$. This implies that  $\stg$ has order $\frac{m^2}{(m,2a)}$. Obviously, each $\sigma_{\tau}(\chi^j,g^k)$ can be expressed as a production of some powers of $\stg$ and $\stx$. Thus we get the desired presentation of $\GO$.
\end{proof}

$\GO$ is actually a metacyclic group, for details, see \cite{metacyclic}. In general, it is not easy to determine the group structure of $\GO$ while in our case $\GO$ can be gotten not so hard.
\begin{prop}
	We have $\GO\cong Z_{(2a,m)} \times Z_{\frac{m^2}{(2a,m)}}$. 
\end{prop}
\begin{proof}
	 It is obvious that $\GO$ is an abelian group and has order $m^2$. By the presentation of $\GO$, the number of generators of $\GO$, must be equal or less than $2$. Thus we may write $\GO \cong Z_{m_1} \times Z_{m_2}$, where $m_1 \mid m_2$.
		 Consider the element $\stg$ and we know that its order is $\frac{m^2}{(2a,m)}$. Hence $\GO$ has a cyclic subgroup $\langle \stg \rangle$ of order $\frac{m^2}{(2a,m)}$.
 If $(2a,m)=1$, then $\stg$ has order $m^2$. So $\GO \cong Z_{m^2}=\langle \stg \rangle$.  Actually, we may regard it as $Z_1 \times Z_{m^2}$ for consistency. 
\par  If $(2a,m)\neq 1$, 
	then $\frac{m^2}{(2a,m)}$ is strict less than $m^2$. 
	We claim that for  arbitrary element $h=\sigma_{\tau}(\chi^i,g^j)$, $ 0\leq i,j <m$, the order  of $h$ will be less than or equal to  $\frac{m^2}{(2a,m)}$. The case $i=j=0$ is trivial and for the case $i\neq 0$ but $j=0$, $ \operatorname{ord}(h)=\frac{m}{(m,i)} \leq m \leq \frac{m^2}{(m,2a)}$. The remaining case is that $j\neq 0$, by direct computation. 
 \begin{align*}
     h^{\frac{m^2}{(m,2aj)}}&=(\stx^{im}\cdot \stg^{jm})^{\frac{m}{(m,2aj)}} \\
     &=(\stx^{2aj})^{\frac{m}{(m,2aj)}}\\
     &=1.
 \end{align*}
 So $\operatorname{ord}(h) \leq\frac{m^2}{(m,2aj)}$. Note that $(m,2aj)\geq (m,2a)$, hence  $\frac{m^2}{(m,2aj)} \leq \frac{m^2}{(m,2a)}$.
 So $\left\langle  \stg \right\rangle $ is a maximal subgroup of $G$. 
	%So $\GO$ has no element of order more than $\frac{m^2}{(2a,m)}$ and there is no real cyclic subgroup of $\GO$ properly including $\langle \stg \rangle$.  
	%We may assume  $\GO \cong Z_{m_1} \times Z_{m_2}$ with $m_1 |m_2$ in this case then $\langle \stg \rangle$ must be isomorphic to  $Z_{m_2}$ since $H$ is a maximal cyclic subgroup of $\GO$ . 
	 %Note that $\GO/\langle \stg \rangle= \overline{\left\langle \stx\right\rangle }$ is a cyclic quotient group  order $(2a,m)$ by the presentation of $\GO$. 
	 Since $\GO \cong Z_{m_1}\times Z_{m_2}$ with $m_1 \mid m_2$, $\left\langle \stg \right\rangle $ must be isomorphic to  $Z_{m_2}$. Hence $Z_{m_1}$ has order $(2a,m)$. So $\GO\cong  Z_{(2a,m)} \times Z_{\frac{m^2}{(2a,m)}}$.
	%	Hence $x, y $ may be written as $x=(1,g)^{a_1}(\chi,1)^{a_2}$ and $y=(1,g)^{b_1}(\chi,1)^{b_2}$ for $0\leq a_i, b_i <m$.  Since $(1,g), x, y$ are generators of $\GO$, then $(1,g)^{k_1}x^{k_2}y^{k_3}=(1,1)$ should imply $k_1\equiv  k_2 \equiv k_3 \equiv 0\  \text{mod} \ m$. 
	%	But we could choose  such a pair$(k_1,k_2,k_3)$ which are not congruent to $0$ for all $k_i$. First, there exists  such that 
\end{proof}
\subsection{A criterion of $3$-coboudaries on abelian groups}
  Let us recall an approach to determine whether a $3$-cohomology on a finite abelian group is nontrivial or not. \par 
Let $H\cong Z_{m_1}\times Z_{m_2} \times \cdots \times Z_{m_n}$ be a finite abelian group and  $(B_{\bullet},\partial_{\bullet})$ be the bar resolution of $H$. By applying $\operatorname{Hom}_{\mathbb{Z}H}(-, \vmathbb{k}^{\times})$ we get a complex $(B^{*}_{\bullet},\partial^{*}_{\bullet})$, where $\vmathbb{k}^{\times}$ is a trivial $H$-module. In \cite{QQG} Section 3, the authors defined another free resolution $(K_{\bullet},d_{\bullet})$ for arbitrary abelian groups and constructed a chain map $F_{\bullet}$ from $(K_{\bullet},d_{\bullet})$ to $(B_{\bullet},\partial_{\bullet})$. For our purpose, we only need the morphism $F_3$, see [\cite{QQG} Lemma 3.9] :
$$\begin{aligned}
	&	F_3: K_3  \rightarrow B_3, \\
	&	\Psi_{r, s, t}  \mapsto\left[g_r, g_s, g_t\right]-\left[g_s, g_r, g_t\right]-\left[g_r, g_t, g_s\right]+\left[g_t, g_r, g_s\right]+\left[g_s, g_t, g_r\right]-\left[g_t, g_s, g_r\right], \\
	&	\Psi_{r, r, s}  \mapsto \sum_{l=0}^{m_r-1}\left(\left[g_r^l, g_r, g_s\right]-\left[g_r^l, g_s, g_r\right]+\left[g_s, g_r^l, g_r\right]\right), \\
	&	\Psi_{r, s, s}  \mapsto \sum_{l=0}^{m_s-1}\left(\left[g_r, g_s^l, g_s\right]-\left[g_s^l, g_r, g_s\right]+\left[g_s^l, g_s, g_r\right]\right), \\
	&	\Psi_{r, r, r}  \mapsto \sum_{l=0}^{m_r-1}\left[g_r, g_r^l, g_r\right],
\end{aligned}
$$
for $1 \leq r<s<t \leq n$, where the symbols like $\Psi_{r, r, r}$ are terms in the resolutions $(K_{\bullet},d_{\bullet})$. 
Moreover, we have the following observation since $F_3^*$ induces an isomorphism between $3$-cohomology groups.
\begin{lemma} \label{lem3.16}
	Let $\phi $ in $(B^{*}_{\bullet},\partial^{*}_{\bullet})$ be a $3$-cocycle. Then $\phi$ is a $3$-coboundary if and only if $F_3^{*}(\phi)$ is a $3$-coboundary.
\end{lemma}
The following lemma provides a criterion for whether a $3$-cochain $f \in \operatorname{Hom}_{\mathbb{Z}H}(K_3,\vmathbb{k}^{\times})$ is $3$-coboundary.
\begin{lemma}[\cite{QQG} Lemma 3.3] \label{lem 3.17}
	The $3$-cochain $f \in \operatorname{Hom}_{\mathbb{Z}H}(K_3, \vmathbb{k}^{\times})$ is $3$-coboundary if and only if for all $1 \leq i <j \leq n$, there are $g_{i,j} \in \vmathbb{k}^{\times}$ such that 
	\begin{equation}\label{3.7}
		f(\Psi_{i,,i,j})=g_{i,j}^{m_i}, \  \ \ f(\Psi_{i,j,j})=g_{i,j}^{-m_j}, \ \ \  \text{and} \ \ f(\Psi_{l,l,l})=1, \ \ f(\Psi_{r,s,t})=1.
	\end{equation}
	for $1 \leq l \leq n $ and $ 1\leq r < s< t \leq n$.
\end{lemma} 
\subsection{The second main result}
In \cite{indicator} , the authors gave a criterion when a twisted quantum double with an abelian cocycle to be genuine.
\begin{lemma}[\cite{indicator} Theorem 4.1, Lemma 4.5]
	Let $G$ be a finite abelian group, and $\omega$ a normalized  abelian $3$-cocycle of $G$. Then $\DG$ is a genuine quasi-Hopf algebra if, and only if $\omega' \in Z^3(\Gamma^{\omega}, \mathbb{C}^{\times})$ is a nontrivial $3$-cocycle of $\Gamma^{\omega}$, where $\omega' \in  Z^3(\Gamma^{\omega},\mathbb{C}^{\times})$ is the inflation of $\omega^{-1}$ along the above map $\Gamma^{\omega} \longrightarrow G$.
\end{lemma}
Now it suffices to  determine whether $\omega'$ is nontrivial on $\GO \cong Z_{(2a,m)} \times Z_{\frac{m^2}{(2a,m)}} $ or not.
Obviously, if $\omega'$ is nontrivial on $Z_{\frac{m^2}{(2a,m)}} $, then $\omega'$ will be nontrivial on $\GO$. Hence we may consider this condition at first. 
\begin{prop}\label{prop 4.8}
	Let $G \cong Z_m$ be a finite cyclic group and $\omega(g^i,g^j,g^k)= \zeta_m^{ai[\frac{j+k}{m}]}$  for $ 1\leq a < m$. If 
	$(m,2a) \nmid (m,a)$, then $\omega'$ is nontrivial on $\GO$.
\end{prop}
\begin{proof}
	Since $\GO$ is the extension of $G$ by $\widehat{G}$. there is a obvious group surjection 
	\begin{align*}
		\pi:  \GO  \longrightarrow Z_{m}: \ &\stx^j \mapsto 1, \\
		&\stg^i \mapsto g^i. 
	\end{align*}
	Hence $\pi^*(\omega^{-1})$ will actually be the restriction of $\omega'$ to $Z_{\frac{m^2}{(2a,m)}}$.\par 
	To show $\pi^*(\omega^{-1})$ is nontrivial, it suffices to show $F_3^*(\pi^*(\omega^{-1}))(\Psi_{1,1,1})$ not equals to $1$ by Lemmas \ref{lem3.16} and \ref{lem 3.17}. By definition of $F_3$, 
	\begin{align*}
		F_3^*(\pi^*(\omega^{-1}))(\Psi_{1,1,1})&=\pi^*(\omega^{-1})(\sum_{l=0}^{\frac{m^2}{(2a,m)}-1}[\stg,\stg^l,\stg])\\
		&=\omega^{-1}(\sum_{l=0}^{\frac{m^2}{(2a,m)}-1}[g,g^l,g])\\
		&=(\zeta_m^{-a})^{\frac{m}{(2a,m)}}.
	\end{align*}
	Note that $(\zeta_m^{-a})^{\frac{m}{(2a,m)}}=1$ if and only if $ \frac{m}{(m,a)}\mid \frac{m}{(m,2a)}$, that is, $(m,2a)$ should divide $(m,a)$. Hence if 
	$(m,2a) \nmid (m,a)$, then $\omega'$ is nontrivial on $Z_{\frac{m^2}{(2a,m)}} $, hence on $\GO$.
\end{proof}
The necessity of Theorem \ref{thm 3.19} is obvious by Proposition \ref{prop 4.8}, since $\omega'$ is a $3$-coboundary on $\GO$ will imply $(m,2a) \mid (m,a)$.\par 
Now we need to deal with the case $(m,2a)\mid (m,a)$. Unfortunately, it is difficult to write down the explicit generator of $Z_{(2a,m)}$. We avoid this difficulty via the following result. By \cite{FQQR2} Lemma 2.16, it's harmless to assume $Z_{(2a,m)}=\langle \stx\cdot \stg^b\rangle=\langle \sigma_{\tau}(\chi,g^b)\rangle$ for $ 0\leq b \leq (2a,m)$. Note that this assumption requires 
$$ m \mid b(2a,m), \ \ \ \text{and} \ \ \ m\mid (2a,m)+2a[\frac{b(2a,m)}{m}].$$ since $\stxg^{(2a,m)}=1$.
All preparations have been done and we are going to  prove Theorem \ref{thm 3.19}.
\begin{proof}[Proof of Theorem \ref{thm 3.19}]
	We only need to show $\omega'$ is a $3$-coboundary on $\GO$ if $(m,2a) \mid (m,a)$.  For consistency, we regard $\stg$ as the first generator and $\stxg$  the second generator. We have already shown $F_3^*(\pi^*(\omega^{-1}))(\Psi_{1,1,1})=1$. The remaining is to verify the condition in Lemma \ref{lem 3.17}. By direct computations, we have 
	\begin{align*}
		F_3^*(\pi^*(\omega^{-1}))(\Psi_{2,2,2})&=\pi^*(\omega^{-1})(\sum_{l=0}^{(2a,m)-1}[\stxg,\stxg^l,\stxg]\\
		&=\omega^{-1}(\sum_{l=0}^{(2a,m)-1}[g^b,(g^b)^l,g^b])\\
		&=\prod_{l=0}^{(2a,m)-1}(\zeta_m^{-a})^{b[\frac{(bl)'+b}{m}]}.
	\end{align*}
	 We have  $m\mid ab $ since $(2a,m)\mid (m,a)$,$ (m,a) \mid a$ and $m \mid (2a,m)b$ by assumption, thus  $F_3^*(\pi^*(\omega^{-1}))(\Psi_{2,2,2})=1$.\par 
	Next we are going to compute 	$F_3^*(\pi^*(\omega^{-1}))(\Psi_{1,2,2})$ and $F_3^*(\pi^*(\omega^{-1}))(\Psi_{1,1,2})$.
	We have
	\begin{align*}
		F_3^*(\pi^*(\omega^{-1}))&(\Psi_{1,1,2})\\
		&=\pi^*(\omega^{-1})(\sum_{l=0}^{\frac{m^2}{(2a,m)}-1}[\stg^l,\stg,\stxg]-[\stg^l,\stxg,\stg]\\&+[\stxg,\stg^l,\stg])\\
		&=\prod_{l=0}^{\frac{m^2}{(2a,m)}-1}\frac{\omega^{-1}(g^{l'},g,g^b)\omega^{-1}(g^b,g^{l'},g)}{\omega^{-1}(g^{l'},g^b,g)}\\
		&=\prod_{l=0}^{\frac{m^2}{(2a,m)}-1}(\zeta_m^{-a})^{b[\frac{l'+1}{m}]}=1. 
	\end{align*}
	since $m\mid ab$ by the analysis above. On the other hand, 
	\begin{align*}
		F_3^*(\pi^*(\omega^{-1}))&(\Psi_{1,2,2})\\ &=\pi^*(\omega^{-1})(\sum_{l=0}^{(2a,m)-1}[\stg,\stxg^l,\stxg]-[\stxg^l,\stg,\stxg]\\&+[(\stxg^l,\stxg,\stg])\\
		&=\prod_{l=0}^{(2a,m)-1}\frac{\omega^{-1}(g,g^{(bl)'},g^b)\omega^{-1}(g^{(bl)'},g^b,g)}{\omega^{-1}(g^{(bl)'},g,g^b)}\\
		&=\prod_{l=0}^{(2a,m)-1}(\zeta_m^{-a})^{[\frac{(bl)'+b}{m}]}.
	\end{align*}
	If $b=0$, then $F_3^*(\pi^*(\omega^{-1}))(\Psi_{1,2,2})=1$. If we take $g_{1,2}=1$, then equation (\ref{3.7}) holds, thus $\omega'$ is a $3$-coboundary. If $b \neq 0$,  then  $(b((2a,m)-1))'+b$ equals $m$ since $m\mid b(2a,m)$ by assumption. Thus $$F_3^*(\pi^*(\omega^{-1}))(\Psi_{1,2,2})=(\zeta_m^{-a})^{\frac{b(2a,m)}{m}}.$$ In this case, take $g_{1,2}=\zeta_m^{\frac{ab}{m}}$. Since $(m,2a) \mid m$ and  $m \mid ab$, we have $$g_{1,2}^{-(2a,m)}=(\zeta_m^{-a})^{\frac{b(2a,m)}{m}}=F_3^*(\pi^*(\omega^{-1}))(\Psi_{1,2,2}),$$
	and 
	$$g_{1,2}^{\frac{m^2}{(2a,m)}}= \zeta_m^{\frac{ab}{m}\frac{m}{(2a,m)}m}=1=F_3^*(\pi^*(\omega^{-1}))(\Psi_{1,1,2}).
	$$ As a result, if $(m,2a) \mid (m,a)$, then $\omega'$ is a $3$-coboundary on $\GO$. Hence $\DG$ is genuine if and only if $(m,2a) \nmid (m,a)$.
\end{proof}
Next we will investigate when  $(m,2a) \nmid (m,a)$ holds. This can provide a more intuitive discrimination.
\begin{thm}
	Let $G \cong Z_m$ be a finite cyclic group and $\omega(g^i,g^j,g^k)= \zeta_m^{ai[\frac{j+k}{m}]}$  for $ 1\leq a < m$. Let $m=2^n \prod_{i}p_i^{a_i}$ and $a=2^{n'} \prod_{j}p_j^{b_j}$ be their prime decomposition, where $n,n' \geq 0$.
	Then $\DG$ is genuine if and only if 
$n' <n$. 
\end{thm}
\begin{proof}
	Suppose $m=2^n \prod_{i}p_i^{a_i}$ and $a=2^{n'} \prod_{j}p_j^{b_j}$ be their prime decomposition. Then $$(m,2a)=(2^n \prod_{i}p_i^{a_i},2^{n'+1} \prod_{j}p_j^{b_j})=(2^n,2^{n'+1})\cdot(\prod_{i}p_i^{a_i},\prod_{j}p_j^{b_j}).$$ and $$(m,a)=(2^n \prod_{i}p_i^{a_i},2^{n'} \prod_{j}p_j^{b_j})=(2^n,2^{n'})\cdot (\prod_{i}p_i^{a_i},\prod_{j}p_j^{b_j}).$$
	Thus $(m,2a) \nmid (m,a)$ if and only if $(2^n,2^{n'+1}) \nmid (2^n, 2^{n'})$. This is equivalent to  $n' <n$.
\end{proof}
\begin{rmk}
\rm	(i) Note that if $m$ is odd, then $\DG$ will never be genuine for arbitrary $0
\leq a<m$. This conclusion is consistent with the \cite{Abcocycle} Theorem 9.4.\par 
	(ii) According to Proposition \ref{Prop3.5}, if $G$ is cyclic, $\DG$ will never be gauge equivalent to $D(G')$ for arbitrary finite group $G'$ by the theory of categorical Morita equivalence, but $\DG$ may be gauge equivalent to a Hopf algebra by Theorem \ref{thm 3.19}. 
\end{rmk}

\section{ Application to the classification of finite-dimensional coquasi-Hopf algebras}
	The purpose of this section is to give a new proof of Proposition 4.1 in \cite{huang2024classification} through applying our previous observations. It should be emphasized that Proposition 4.1  plays the key role in that paper and the original proof relays on heavy computations. To do that, we firstly prove that all Nichols algebra generated by three pairwise non-isomorphic Yetter-Drinfeld modules over $D_8$ are infinite-dimensional which seems has its independent interest. 
\subsection{Classification of finite-dimensional Nichols algebras generated by irreducible Yetter-Drinfeld modules over $D_8$}
We first recall the basic notation of irreducible Yetter-Drinfeld modules over groups. Let $G$ be a finite group, $\mathcal{O}$ a conjugacy class of $G$ , $s \in \mathcal{O}$ fixed, $(\rho,V)$ an irreducible representation of $G^s$, where $G^s$ is the centralizer of $s$ in $G$. Let $t_1=s,...,t_M$ be a numeration of $\mathcal{O}$ and let $g_i \in G$ such that $g_isg_i^{-1}=t_i$ for all $1 \leq i \leq M$.  Then the corresponding irreducible Yetter-Drinfeld module $M(\mathcal{O},\rho)$ is defined as follows: As a space, it just $\bigoplus_{1\leq i \leq M}g_i \otimes V$. Let $g_iv:=g_i \otimes v \in M(\mathcal{O},\rho)$, $1\leq i \leq M$, $v \in V$. If $v \in V$ and $1\leq i \leq M$, then the coaction and the action of $g\in G$ are given by
$$ \delta(g_iv)=t_i\otimes g_iv, \ \ \ \ \ g\rhd(g_iv)=g_j(\gamma\circ v),$$
where $gg_i=g_j\gamma$ and $\gamma \circ v=\rho(\gamma)(v)$ for some $1\leq j \leq M$,$\gamma \in G^s$. The Yetter-Drinfeld module $M(\mathcal{O},\rho)$ is a braided vector space with braiding given by
$$c(g_iv \otimes g_jw)=t_i\rhd(g_jw)\otimes g_iv=g_h(\gamma\circ v)\otimes g_iv$$
for any $1\leq i,j \leq M$, $v,w \in V$, where $t_ig_j=g_h\gamma$ for unique $h$, $1\leq h\leq M$ and $\gamma \in G^s$.

 Next, we describe the well-known classification result of finite-dimensional Nichols algebras generated by irreducible Yetter-Drinfeld modules over $D_8$. Recall that the dihedral group  $D_8$ is generated by $x$ and $y$ with the following presentation 
$$\left\langle x,y\mid y^2=1=x^4, yxy=x^{-1}\right\rangle $$
 and let $\chi$ be a character of $\left\langle x\right\rangle$ such that $\chi(x)=\omega$ is a primitive $4$-th root of unity.
\begin{lemma}[\cite{D4} Theorem 3.1]
	Let $M(\mathcal{O},\rho)$ be the irreducible Yetter-Drinfeld module over $D_8$ corresponding to a pair $(\mathcal{O}, \rho)$. Assume that its Nichols algebra is finite-dimensional, then $(\mathcal{O}, \rho)$ is one of the following:
 \begin{itemize}
     \item[\emph{(i)}] $(\mathcal{O}_{x^2},\rho)$, where $\rho \in  \widehat{D_8}$ satisfies $\rho(x^2)=1$. 
\item[\emph{(ii)}] $(\mathcal{O}_{x^{h}},\chi^j)$, where $h=1 \ or \ 3$, and $\omega^{hj}=-1$.
\item[\emph{(iii)}] $(\mathcal{O}_{y},\operatorname{sgn} \otimes \operatorname{sgn})$ or $(\mathcal{O}_{y},\operatorname{sgn} \otimes \varepsilon)$, where $\operatorname{sgn} \otimes \operatorname{sgn}$, $ \operatorname{sgn} \otimes \operatorname{sgn} \in \widehat{D_8^y}$, $D_8^y=\left\langle y\right\rangle \oplus \left\langle x^2\right\rangle \cong Z_2 \times Z_2.$
\item[\emph{(iv)}]  $(\mathcal{O}_{xy},\operatorname{sgn} \otimes \operatorname{sgn})$ or $(\mathcal{O}_{xy},\operatorname{sgn} \otimes \varepsilon)$, where $\operatorname{sgn} \otimes \operatorname{sgn}$, $ \operatorname{sgn} \otimes \operatorname{sgn} \in \widehat{D_8^{xy}}$, $D_8^{xy}=\left\langle xy\right\rangle \oplus \left\langle x^2\right\rangle \cong Z_2 \times Z_2.$ 
 \end{itemize}
\end{lemma}
\begin{rmk}
	\rm (i) In all above cases, $
	\operatorname{dim}M(\mathcal{O},\rho)=2$ and $
	\operatorname{dim}{\mathcal{B}}(\mathcal{O},\rho)=4$. \par 
	(ii) It is obviously that $$M(\mathcal{O}_x,\chi) \cong M(\mathcal{O}_{x^3},\chi^3)$$ as irreducible Yetter-Drinfeld modules. Meanwhile, there are isomorphisms of braided vector spaces 
	$$ M(\mathcal{O}_{y},\operatorname{sgn} \otimes \operatorname{sgn}) \cong M(\mathcal{O}_{xy},\operatorname{sgn} \otimes \operatorname{sgn}),$$
	$$M(\mathcal{O}_{y},\operatorname{sgn} \otimes \varepsilon) \cong M(\mathcal{O}_{xy},\operatorname{sgn} \otimes \varepsilon).$$
\end{rmk}
 \subsection{ Nichols algebras  over $D_8$ of rank $3 $  }
In this subsection, we will prove all Nichols algebras generated by three pairwise nonisomorphic Yetter-Drinfeld modules over $D_8$  are infinite-dimensional. Our main ingredient is  generalized Cartan matrix and Heckenberger's classification of finite-dimensional Nichols algebra of rank $\geq 3$. We first recall the definition of the Cartan matrix. We assume $I$ is a finite non-abelian group in this subsection.
\begin{definition}
	Let $\II$ be the Yetter-Drinfeld module category over $I$ and $\theta \in \mathbb{N}$ and $I=\left\lbrace 1,...,\theta\right\rbrace $. For $N=(N_1,N_2,...,N_{\theta})$ where $N_i$ are simple Yetter-Drinfeld module for all $i$, let 
	$$
	a_{i j}^N= \begin{cases}-\infty & \text { if }\left(\operatorname{ad} N_i\right)^m\left(N_j\right) \neq 0 \text { for all } m \geq 0, \\ -\sup \left\{m \in \mathbb{N}_0:\left(\operatorname{ad} N_i\right)^m\left(N_j\right) \neq 0\right\} & \text { otherwise }\end{cases}
	$$
	for all $i \in I$ and $ j\in I \setminus \left\lbrace i\right\rbrace$. Moreover, let $a_{ii}^N=2$ for all $i \in I$. Then $A^N=(a_{ij}^N)_{i,j \in I}$ is called the generalized Cartan matrix of N.
\end{definition}
So far, the classification of finite-dimensional  Nichols algebra in usual Yetter-Drinfeld module category have achieved many progression. In \cite{Hec},
Heckenberger has classified all finite-dimensional Nichols algebra over a non-abelian group of rank $\geq 3$.  Let's give a brief introduction.

\begin{definition} [\cite{Hec} Definition 2.1]
	Let $\theta \in \mathbb{N}$, $M=(M_1,M_2,...,M_{\theta}) \in \II$ with each $M_i$ simple  is called braid-indecomposable if there exists no decomposition $M' \oplus M''$ of $\bigoplus_{i=1}^{\theta} M_i$ with $M', M'' \neq 0$ such that $(\operatorname{id}-\operatorname{c}^2)(M' \otimes M'') = 0$
\end{definition}
\begin{definition} [\cite{Hec} Definition 2.2]\label{def4.13}
	Let $\theta \in \mathbb{N}$, $M=(M_1,M_2,...,M_{\theta}) \in \II$ with each $M_i$ simple. Let $A=(a_{ij})$ be the generalized Cartan matrix of $M$, we say $M$ has a skeleton if:\par 
$\operatorname{(1)}$ for all $1 \leq i \leq \theta$, there exists  $s_i \in \operatorname{supp} M_i$, and $\sigma_i \in \widehat{G^{s_i}}$ such that $M_i \cong M(\mathcal{O}_{s_i}, \sigma_i),$ and \par 
$\operatorname{(2)}$ for all $1\leq i <j\leq \theta$ with $a_{ij}\neq 0$, at least one of $a_{ij}, a_{ji}$ is $-1$.\end{definition}
In this case the skeleton of $M$ is a partially oriented partially labeled loopless graph with $\theta$
vertices with the following properties:\par
$\bullet$ For all $1 \leq i \leq \theta$, the $i$-th vertex is symbolized by $\left|\operatorname{supp} M_i\right|=\operatorname{dim} M_i$ points. If $\operatorname{dim} M_i=1$, then the vertex is labeled by $\sigma_i\left(s_i\right)$. If $\operatorname{dim} M_i=2$ and there is an additional restriction on $p=\sigma_i\left(s_i^{\prime} s_i^{-1}\right)$, where $\operatorname{supp} M_i=\left\{s_i, s_i^{\prime}\right\}$, then the $i$-th vertex is labeled by $(p)$. Otherwise there is no label.
\par $\bullet$ For all $i, j \in\{1, \ldots, \theta\}$ with $i \neq j$ there are $a_{i j} a_{j i}$ edges between the $i$-th and $j$-th vertex. The edge is oriented towards $j$ if and only if $a_{i j}=-1, a_{j i}<-1$.
\par   $\bullet$ Let $1 \leq i<j \leq \theta$ with $a_{i j}<0$. If $\operatorname{supp} M_i$ and $\operatorname{supp} M_j$ commute, then the connection between the $i$-th and $j$-th vertex consists of continuous lines. Otherwise the connection consists of dashed lines. The connection is labeled with $\sigma_i(s_j) \sigma_j(s_i)$ if $\operatorname{dim} M_i=1$ or $\operatorname{dim} M_j=1$, and otherwise it is not labeled.

The next Theorem gives a criterion to determine when $\mathcal{B}(M)\in \II$ is finite-dimensional.
\begin{thm} [\cite{Hec} Theorem 2.5]\label{thm4.3}
	Let $\theta \in \mathbb{N}_{\geq 3}$. Let $I$ be a non-abelian group and $M=(M_1,M_2,...,M_{\theta})$ with each $M_i$ simple and $\operatorname{supp}M$ generates $I$. Assume that $M$ is braid-indecomposable. Then the following are equivalent:\par
$\operatorname{(1)}$  $M$ has a skeleton of finite type.\par
$\operatorname{(2)}$ $\mathcal{\mathcal{B}}(M)$ is finite-dimensional.\par
$\operatorname{(3)}$  $M$ admits all reflections and the Weyl groupoid $\mathcal{W}(M)$ of $M$ is finite.
\end{thm}
A complete classification result of skeletons of ﬁnite type with at least three vertices over arbitrary field is  given  simultaneously, see \cite{Hec}.

Let us return to the dihedral group $D_8$ case. There are six nonisomorphic irreducible Yetter-Drinfeld modules over $D_8$. For simplicity, denote
$M_1=M(\mathcal{O}_{x^2},\rho)=\operatorname{span}\left\lbrace 1u_1, 1u_2\right\rbrace $, $M_2=M(\mathcal{O}_{x},\chi)=\operatorname{span}\left\lbrace 1v,yv\right\rbrace $,  
$M_3=M(\mathcal{O}_{y},\operatorname{sgn} \otimes \operatorname{sgn})=\operatorname{span}\left\lbrace 1w_1,xw_1\right\rbrace $,
$M_4=M(\mathcal{O}_{y},\operatorname{sgn} \otimes \varepsilon)=\operatorname{span}\left\lbrace 1w_2,xw_2\right\rbrace $,
$M_5=M(\mathcal{O}_{xy},\operatorname{sgn} \otimes \operatorname{sgn})=\operatorname{span}\left\lbrace 1w_3,xw_3\right\rbrace $, 
$M_6=M(\mathcal{O}_{xy},\operatorname{sgn} \otimes \varepsilon)=\operatorname{span}\left\lbrace 1w_4,xw_4\right\rbrace $.
For simplicity, we denote $S:=\left\lbrace M=(M_i,M_j,M_k) \mid 1 \leq  i < j < k
\leq 6\right\rbrace $.
Now we are going to state the main result of this subsection. Actually, it just results from direct computations.
\begin{thm} \label{thm5.7}
	The Nichols algebra $\mathcal{B}(M)=\mathcal{B}(M_i\oplus M_j \oplus M_k)$ is infinite-dimensional for all $1 \leq  i < j < k
	\leq 6$.
\end{thm}
This theorem relies on the following lemmas.  We deal with the cases when $\operatorname{supp}(M)$ is an abelian group at first. In these cases, $M$ can be reduced to a diagonal type  Yetter-Drinfeld module  over $\operatorname{supp}(M)$.
\begin{lemma}\label{lem5.8}
	Let $M=(M_i,M_j,M_k)$, where $1 \leq  i < j < k
	\leq 6$. Suppose $M \in S_1:=\left\lbrace (M_1,M_2,M_k)  \mid  3\leq k \leq 6 \right\rbrace $. Then $\mathcal{B}(M)$ is infinite-dimensional. 
\end{lemma}
\begin{proof}
 Note $\operatorname{supp}(M_1 \oplus M_2)=\left\langle  x\right\rangle \cong Z_4$.  By restriction, $\mathcal{B}(M_1\oplus M_2) \in {^{Z_4}_{Z_4}\mathcal{YD}}$ is of diagonal type. Hence
	we choose a new basis of $M_1$ by setting $t_1=1u_1+i(1u_2)$ and $t_2=1u_1-i(1u_2)$.  Then $M_1 \oplus M_2=\operatorname{span}\left\lbrace t_1,t_2,1v,yv\right\rbrace $.  Direct computation gives the braiding matrix:
\[	\begin{pmatrix}
		-1&-1&1&1 \\
		-1 &-1 &1&1 \\
		-i& i & -1 &-1 \\
		i &-i &-1&-1
	\end{pmatrix}.\]

The corresponding generalized Dynkin diagram is of the form
	
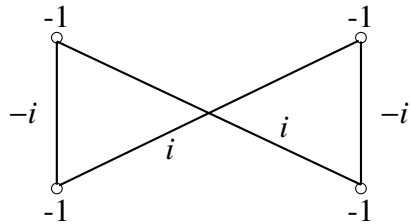
\begin{center}
\begin{tikzpicture}
	\draw (0,0) circle (2pt) node[anchor=north]{-1} ;
		\draw (0,2) circle (2pt)node[anchor=south]{-1};
		\draw [black] (4,0) circle (2pt)node[anchor=north]{-1};
		\draw [black] (4,2) circle (2pt) node[anchor=south]{-1};
		\draw[thick, -] (0,0.05).. controls (0,0.1) and (0,1.9) .. node[anchor=south]{}(0,1.95) ;
		\draw (1.5,0.8) circle (0pt) node[anchor=north]{$i$} ;
			\draw (-0.1,1) circle (0pt) node[anchor=east]{$-i$} ;
					\draw (4.1,1) circle (0pt) node[anchor=west]{$-i$} ;
			\draw (3,1.1) circle (0pt) node[anchor=north]{$i$} ;
		\draw[thick, -] (0.03,0.04).. controls (0.03,0.04) and (3.97,1.96) .. node[anchor= east]{}(3.97,1.96) ;
		\draw[thick, -] (0,1.95).. controls (0,1.95) and (4,0.05) .. node[anchor=north]{}(4,0.05) ;
		\draw[thick, -] (4,1.95).. controls (4,1.95) and (4,0.05) .. node[anchor=south]{}(4,0.05) ;
	\end{tikzpicture}
\end{center} 
which does not appear in the classification of  arithmetic root system \cite{ARS}. So $\mathcal{B}(M_1\oplus M_2)$ is infinite-dimensional, hence $\mathcal{B}(M)$ is infinite-dimensional. 

\end{proof}
\begin{lemma}\label{lem5.9}
	If $M\in S_2:=\left\lbrace (M_1,M_3,M_4),\ (M_1,M_5,M_6)\right\rbrace $, then $\mathcal{B}(M)$ is infinite-dimensional.
\end{lemma}
\begin{proof}
	We will prove $\mathcal{B}(M)=\mathcal{B}(M_1 \oplus M_3 \oplus M_4)$ is infinite-dimensional, another case is similar. Note that $\operatorname{supp}(M_1 \oplus M_4 \oplus M_5) \cong Z_2 \times Z_2=\left\langle  y \right\rangle \times \left\langle x^2\right\rangle $
	By restriction, $\mathcal{B}(M) \in {^{Z_2\times Z_2}_{Z_2\times Z_2}\mathcal{YD}}$ is of diagonal type, We choose a new basis of $M_1$ via $t_1=1u_1+1u_2$, $t_2=1u_1-1u_2$. Then $M =\operatorname{span}\left\lbrace t_1,t_2,1w_1,xw_1,1w_2,xw_2\right\rbrace $ by direct  computation, the corresponding braiding matrix is 
	\[	\begin{pmatrix}
		-1&-1&1&1 &1 &1\\
		-1 &-1 &-1&- 1&1 &1 \\
		-1& 1 & -1 &1 & -1 &-1 \\
		1&-1 &1&-1 &-1 &-1 \\
		-1& 1 &-1 &1 &-1 &-1 \\
		1& -1 &1 &-1 &-1 &-1
	\end{pmatrix}.\]
	The corresponding generalized Dynkin diagram is of the form
	
		
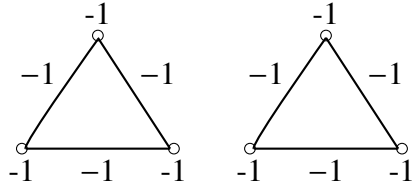
\begin{center}
		\begin{tikzpicture}
			\draw (0,0) circle (2pt) node[anchor=north]{-1} ;
			\draw (2,0) circle (2pt)node[anchor=north]{-1};
					\draw (1,1.5) circle (2pt)node[anchor=south]{-1};
			
			\draw[thick, -] (0.05,0).. controls (0.05,0) and (1.95,0) .. node[anchor=south]{}(1.95,0)  ;
			\draw[thick, -] (0.05,0).. controls (0,0.05) and (1,1.45) .. node[anchor=south]{}(1,1.45)  ;
				\draw[thick, -] (1.95,0).. controls (1.95,0) and (1,1.47) .. node[anchor=south]{}(1,1.47)  ;
			\draw (1,0) circle (0pt) node[anchor=north]{$-1$} ;
			\draw (0.6,1) circle (0pt) node[anchor=east]{$-1$} ;
			\draw (1.4,1) circle (0pt) node[anchor=west]{$-1$} ;
	\draw (3,0) circle (2pt) node[anchor=north]{-1} ;
	\draw (5,0) circle (2pt)node[anchor=north]{-1};
	\draw (4,1.5) circle (2pt)node[anchor=south]{-1};
	
	\draw[thick, -] (3.05,0).. controls (3.05,0) and (4.95,0) .. node[anchor=south]{}(4.95,0)  ;
	\draw[thick, -] (3.05,0).. controls (3,0.05) and (4,1.47) .. node[anchor=south]{}(4,1.47)  ;
	\draw[thick, -] (4.95,0).. controls (4.95,0) and (4,1.47) .. node[anchor=south]{}(4,1.47)  ;
	\draw (4,0) circle (0pt) node[anchor=north]{$-1$} ;
	\draw (3.6,1) circle (0pt) node[anchor=east]{$-1$} ;
	\draw (4.4,1) circle (0pt) node[anchor=west]{$-1$} ;
	
		\end{tikzpicture}
	\end{center} 
which does not  appear in the classification of  arithmetic root system \cite{ARS}. So $\mathcal{B}(M)$ is infinite-dimensional.
\end{proof}
We are going to deal with the cases which not appear in Lemmas \ref{lem5.8} and \ref {lem5.9}. It is obviously that $\operatorname{supp}(M)=D_8$ in these cases. We are going to use Theorem \ref{thm4.3} to show $\mathcal{B}(M)$ are all infinite-dimensional. 
\begin{lemma}
	Suppose $M=(M_i,M_j,M_k)\in S \setminus {S_1 \cup S_2}$,  then $M$ is braid-indecomposable. 
\end{lemma}
\begin{proof}
	It is not difficult to observe that as long as for any $i,j$ $1\leq i <j \leq 6$, $(\operatorname{id}-c^2)(M_i \otimes M_j) \neq 0$. Then $M$ is braid-indecomposable for all $M \in S \setminus S_1 \cup S_2$. In particular, we will not consider  braid-indecomposability of $M_1 \oplus M_2$ since $\mathcal{B}(M_1 \oplus M_2)$ is infinite-dimensional by Lemma \ref{lem5.8}. We will compute one case as an example and list  complete situations in the following table.\par 
	We are going to show $M_1 \otimes M_3$ is  braid-indecomposable.
	Choose $1u_1 \otimes 1w_1 \in M_1 \otimes M_3$. Note $\delta(1u_1)= x^2 \otimes 1u_1$ and $\delta(1w_1)=y \otimes 1w_1$. Then
	$c(1u_1 \otimes 1w_1)=x^2\rhd (1w_1) \otimes 1u_1=-1w_1 \otimes 1u_1$, and $c(-1w_1 \otimes 1u_1)=-y \rhd (1u_1)\otimes 1w_1=1u_2\otimes 1w_1$. Hence $$(\operatorname{id}-c^2)(1u_1 \otimes 1w_1)=1u_1\otimes 1w_1-1u_2\otimes 1w_1 \neq 0.$$  
		\begin{table}[htbp]  % 位置选项，‌表示表格可以放在这里、‌这里、‌页顶或页底
			\centering  % 居中显示
			\caption{Braiding-Indecomposability of $M$}  % 表格标题
			\begin{tabular}{|c|c|c|c|}  % 表格内容
				\hline
				$M_i \otimes M_j$ &   $ x\otimes y \in M_i \otimes M_j$,  s.t. $(\operatorname{id}-c^2)(x\otimes y)\neq 0.$& $M_i \otimes M_j$  & $ x\otimes y \in M_i \otimes M_j$,  s.t. $(\operatorname{id}-c^2)(x\otimes y)\neq 0.$ \\ \hline
				$M_1\otimes M_3$ & $1u_1 \otimes 1w_1$ & $M_1 \otimes M_4$ & $1u_1 \otimes xw_2$
				 \\
				\hline
				$M_1\otimes M_5$ & $1u_1 \otimes 1w_3$ & $M_1 \otimes M_6$ & $1u_1 \otimes xw_4$ \\ \hline
				$M_2 \otimes M_3$ & $1v \otimes xw_1$ & $M_2 \otimes M_4$ & $1v \otimes xw_2$ \\ \hline
				$M_2 \otimes M_5$ & $1v \otimes 1w_3$ & $M_2 \otimes M_6$ & $1v \otimes 1w_4$ \\ \hline 
				$M_3 \otimes M_4$ & $1w_1 \otimes xw_2$ & $M_3 \otimes M_5$ & $1w_1 \otimes xw_3$ \\ \hline 
				$M_3 \otimes M_6$ & $1w_1 \otimes xw_4$ & $M_4 \otimes M_5$ & $1w_2 \otimes xw_3$ \\ \hline
				$ M_4 \otimes M_6$ & $1w_2 \otimes xw_4$ & $M_5 \otimes M_6$ & $1w_3 \otimes xw_4$ \\ \hline
			\end{tabular}
		\end{table}
		
\end{proof}
\begin{prop}\label{prop5.11}
	Suppose $M=(M_i,M_j,M_k)\in S \setminus {S_1 \cup S_2}$,  then $\mathcal{B}(M)= \mathcal{B}(M_i \oplus M_j \oplus M_k)$ is infinite-dimensional.
\end{prop}
\begin{proof}
	We choose some cases to calculate since they are similar. The key is to find out the generalized Cartan matrix for each $M$, then draw the corresponding skeleton and apply the Theorem \ref{thm4.3} finally. \par 
	Take $M=(M_1, M_3,M_5)$,
	we first calculate the number $a^M_{13}$. Note that  $\operatorname{ad}_{1u_1}(1w_1)=1u_1\cdot 1w_1+1w_1\cdot 1u_1$, $\operatorname{ad}_{1u_1}(xw_2)= 1u_1\cdot xw_1+xw_1\cdot 1u_1$, $\operatorname{ad}_{1u_2}(1w_1)=1u_2\cdot 1w_1+1w_1\cdot 1u_2$, and $\operatorname{ad}_{1u_2}(xw_1)=1u_2\cdot xw_1+xw_1\cdot 1u_2$. Then we take coproduct of these elements. 
	\begin{align*}
		\Delta(\operatorname{ad}_{1u_1}(1w_1))&=1\otimes \operatorname{ad}_{1u_1}(1w_1)+\operatorname{ad}_{1u_1}(1w_1)\otimes 1 +1u_1 \otimes 1w_1-1u_2\otimes 1w_1, \\
		\Delta(\operatorname{ad}_{1u_1}(xw_1))&= 1\otimes \operatorname{ad}_{1u_1}(xw_1)+ \operatorname{ad}_{1u_1}(xw_1) \otimes 1+1u_1 \otimes xw_1 +1u_2 \otimes xw_1, \\
			\Delta(\operatorname{ad}_{1u_2}(1w_1))&= 1\otimes \operatorname{ad}_{1u_2}(1w_1)+ \operatorname{ad}_{1u_2}(1w_1) \otimes 1+1u_2 \otimes 1w_1 -1u_1 \otimes 1w_1, \\
			\Delta(\operatorname{ad}_{1u_2}(xw_1))&=1\otimes\operatorname{ad}_{1u_2}(xw_1)+\operatorname{ad}_{1u_2}(xw_1)\otimes 1+1u_2  \otimes xw_1-1u_1 \otimes xw_1.
	\end{align*}
Obviously, $\operatorname{ad}_{1u_1}(1w_1)+\operatorname{ad}_{1u_2}(1w_1)=0$ and $\operatorname{ad}_{1u_1}(xw_1)+\operatorname{ad}_{1u_2}(xw_1)=0 $ by their coproduct. Hence $\operatorname{ad}_{M_1}(M_3)=\operatorname{span} \left\lbrace \operatorname{ad}_{1u_1}(1w_1), \operatorname{ad}_{1u_2}(xw_1)\right\rbrace $. \par 
Next, since $1u_1^2=1u_2^2=1w_1^2=xw_1^2=0$, we have $\operatorname{ad}_{1u_1}(\operatorname{ad}_{1u_1}(1w_1))=1u_1 \cdot (1u_1\cdot 1w_1+1w_1\cdot 1u_1)-x^2 \rhd (1u_1\cdot 1w_1+1w_1\cdot 1u_1) 1u_1=0$ as well as $\operatorname{ad}_{1u_2}(\operatorname{ad}_{1u_2}(xw_1))=0$. Moreover 
\begin{align*}
	\operatorname{ad}_{1u_2}(\operatorname{ad}_{1u_1}(1w_1))&=1u_2\cdot (1u_1\cdot 1w_1+1w_1\cdot 1u_1)-x^2\rhd (1u_1\cdot 1w_1+1w_1\cdot 1u_1)\cdot 1u_2 \\
&	=1u_2\cdot 1u_1 \cdot 1w_1 +1u_2 \cdot 1w_1 \cdot 1u_1-1u_1 \cdot 1w_1 \cdot 1u_2-1w_1 \cdot 1u_1 \cdot 1u_2  \\
&=1u_2\cdot (-1u_2\cdot 1w_1-1w_1\cdot 1u_2)- (-1u_2\cdot 1w_1-1w_1\cdot 1u_2)\cdot 1u_2=0.
\end{align*}
where the last equation we use the fact that $\operatorname{ad}_{1u_1}(1w_1)+\operatorname{ad}_{1u_2}(1w_1)=0$.We can prove 	$\operatorname{ad}_{1u_1}(\operatorname{ad}_{1u_2}(xw_1))=0$ similarly. Thus $\operatorname{ad}^2_{M_1}(M_3)=0$. \par 
Using the same method, we can prove $\operatorname{ad}_{M_1}(M_5) \neq 0$ and  $\operatorname{ad}_{M_3}(M_5) \neq 0$. But $\operatorname{ad}^2_{M_1}(M_5)=\operatorname{ad}^2_{M_3}(M_5)=0$. Hence $M=(M_1,M_3,M_5)$ has Cartan matrix 	 $\begin{pmatrix}
	2 & -1 & -1 \\
	-1 & 2 & -1 \\ 
	-1 & -1 & 2 
\end{pmatrix}$. 
\par It is not surprising that for all $M \in S \setminus S_1 \cup S_2$, the Cartan matrix of $M$ are all $\begin{pmatrix}
	2 & -1 & -1 \\
	-1 & 2 & -1 \\ 
	-1 & -1 & 2 
\end{pmatrix}$, because their Yetter-Drinfeld module structures are similar.  We omit the proof for simplicity.\par 
Although they have the same Cartan matrix, the corresponding skeletons may be different. \par 
If $(i,j,k) \in \left\lbrace (1,3,5), \ (1,3,6), \ (1,4,5), \ (1,4,6)\right\rbrace $, the corresponding skeleton will be the first picture. \par 
If $(i,j,k) \in \left\lbrace (2,3,4), \ (3,4,5), \ (3,4,6), \ (3,5,6),\ (4,5,6), \ (2,5,6)\right\rbrace $, the corresponding skeleton will be the second picture. \par 
If $(i,j,k) \in \left\lbrace (2,3,5), \ (2,3,6), \ (2,4,5), \ (2,4,6)\right\rbrace $, the corresponding skeleton will be the third picture. 
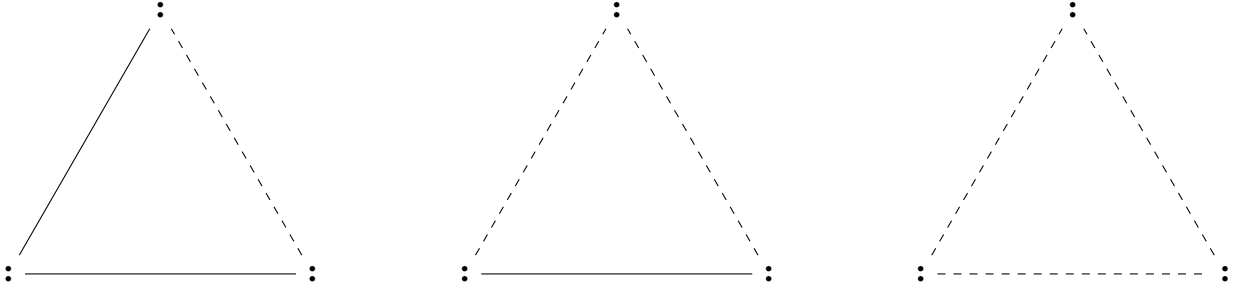
\begin{center}
	\begin{tikzpicture}
		\node (A) at (-2,0) {\textbf{:}};
		\node (B) at (2,0) {\textbf{:}};
		\node (C) at (0,3.5) {\textbf{:}};
		\node (D) at (-8,0) {\textbf{:}};
		\node (E) at (-4,0) {\textbf{:}};
		\node (F) at (-6,3.5) {\textbf{:}};
		\node (G) at (4,0) {\textbf{:}};
		\node (H) at (8,0) {\textbf{:}};
		\node (I) at (6,3.5) {\textbf{:}};
		\draw[-] (A) --node [above ] {} (B);
		\draw[dashed] (A) --node [ right] {} (C);	
		\draw[dashed] (B) --node [above ] {} (C);	
		\draw[-] (D) --node [above ] {} (E);
		\draw[-] (D) --node [ right] {} (F);	
		\draw[dashed] (E) --node [above ] {} (F);
		\draw[dashed] (G) --node [above ] {} (H);
		\draw[dashed] (G) --node [ right] {} (I);	
		\draw[dashed] (H) --node [above ] {} (I);
	\end{tikzpicture}
\end{center}
All  skeletons above don't appear  in \cite{Hec} Figure$2.1$, hence by Theorem \ref{thm4.3}, $\operatorname{dim}(\mathcal{B}(M))=\infty$ for all $M \in S\setminus S_1 \cup S_2$.
\end{proof}
\begin{proof}[Proof of Theorem \ref{thm5.7}]
	It is direct from Lemma \ref{lem5.8}, \ref{lem5.9} and Proposition \ref{prop5.11}.
\end{proof}
	\subsection{An invariant preserved by gauge equivalence}
%For completeness, we will recall the basic notation of coquasi-Hopf algebras and related concepts briefly. But for simplicity, some details will be omitted.\par 
 By definition, a coquasi-Hopf algebra is exactly the dual notion of a Drinfeld’s quasi-Hopf algebra
 \cite{Drinfeld}. One may refer \cite{QQG} Section 2 for explicit definition, examples and related notions such as the category of Yetter-Drinfeld module (which is denoted by $\GG$) over the coquasi-Hopf algebra $(\vmathbb{k}G,\omega )$, the Nichols algebra in this category etc.
Now let $G$ be an abelian group with a nontrivial $3$-cocycle $\omega$ and $H$ is a finite group. Denoting $\GGd$ the full subcategory of $\GG$consisting of all finite-dimensional twisted Yetter-Drinfeld modules. Suppose $F:\  \GGd \longrightarrow \HHd$ is an equivalence of fusion categories. The following results seem well-known, but we can't find suitable reference. For the convenience of readers, we write them out. 
%For the subsequent proof of Theorem \ref{thm1.3}, we want to show $F$ preserves the  dimension of objects as vector spaces and $F$ maps finite-dimensional twisted Nichols algebras to finite-dimensional usual Nichols algebras right now.
\begin{lemma}\label{prop4.8}
	For each $X \in \GGd$, we have the following equations with respect to dimensions:  $$ \operatorname{dim}_{\vmathbb{k}}(X)=\operatorname{FPdim}(X)=\operatorname{FPdim}(F(X))=\operatorname{dim}_{\vmathbb{k}}(F(X)).$$
\end{lemma}
\begin{proof}
	There's an equivalence of fusion categories:  $\GGd \cong \mathcal{Z}(\VG) \cong  \operatorname{Rep}(D^{\omega}(G))$. By \cite{tensor} Example 5.13.8,  $\operatorname{dim}_{\vmathbb{k}}(X)=\operatorname{FPdim}(X)$ for all $X \in \, \GGd$. \par
	Note $\GGd$ and $\HH_{\operatorname{fd}}$ are all fusion categories and $F$ is a tensor functor, 
	then by \cite{tensor} Proposition 4.5.7
	$$\operatorname{FPdim}_{\GGd}(X)=\operatorname{FPdim}_{\HH_{\operatorname{fd}}}(F(X)).$$
	The last equation $\operatorname{FPdim}(F(X))=\operatorname{dim}_{\vmathbb{k}}(F(X))$ can be obtained via using the same result as the first equation.
\end{proof}
\begin{prop}\label{prop4.7}
	Suppose $F:\  \GGd \longrightarrow \HHd$ is a tensor equivalence. Then $F$ maps a Nichols algebra in $\GGd$ to a Nichols algebra in $\HHd$.
\end{prop}
\begin{proof}
	Although we deal with finite-dimensional Nichols algebras, but by definition they are quotient of infinite-dimensional objects in $\GG$, so we should extend $F$ to $\GG$ at first.\par  For all $X,Y \in \GGd$, define $\widetilde{F}(X)=F(X)$ and $\widetilde{F}(f)=F(f):F(X)\rightarrow F(Y)$. Now suppose $X \in \GG$ but $X \notin \GGd$. Since $\GG$ and $\HH$ are semisimple categories and the Grothendieck ring of both categories are finite. $X$ will be direct sum of simple objects $X=\bigoplus\limits_{i \in I}X_i$, where $X_i$ are simple objects.  Each $X_i$ belongs to  $\GGd$, because all simple objects in $\GG$ are finite-dimensional.
	Then $\widetilde{F}(X)$ may be defined via 
	$$ \widetilde{F}(X):= \bigoplus_{i \in I} F(X_i).$$
	For $f: X\rightarrow Y$ be a morphism in $\GG$, $\widetilde{F}(f)$ may be defined as $\bigoplus\limits_{(i,j)\in I\times I}\operatorname{Hom}(X_i,Y_j)$. Obviously, $\widetilde{F}$ preserves composition of morphisms, identity morphism and finite direct sums. Thus $\widetilde{F}$ is an additive functor. Now we are going to show $\widetilde{F}$ is a tensor functor. We have the following isomorphism: 
	\begin{align*}
		\widetilde{F}(X\otimes Y)\cong \widetilde{F} (\underset{i\in I}{\oplus}X_i \otimes\underset{j\in I}{\oplus}Y_j) & \cong \widetilde{F}(\underset{i\in I}{\oplus}\underset{j\in I}{\oplus} X_i \otimes Y_j) 
		\\ &= \underset{i\in I}{\oplus}\underset{j\in I}{\oplus}F(X_i\otimes Y_j)\cong \underset{i\in I}{\oplus}\underset{j\in I}{\oplus} F(X_i)\otimes F(Y_j) \cong \widetilde F(X)\otimes \widetilde{F}(Y).
	\end{align*}
	and $F(\vmathbb{k}) \cong \vmathbb{k}$. The tensor structure $\widetilde{F}_{X,Y}: \widetilde{F}(X\otimes Y)\overset{\cong}{\rightarrow} \widetilde{F}(X)\otimes \widetilde{F}(Y)$ satisfies hexagon diagrams since the tensor structure $J$ of $F$ does. Thus $\widetilde{F}$ is a tensor functor.  Moreover, since $F$ is a tensor equivalence, it has an inverse functor $F^{-1}: \HHd \rightarrow \GGd$. Using this functor with similar method, we can prove $\widetilde{F}$ is a tensor equivalence as well. \par 
	Now let $V \in \GGd$ be finite-dimensional and $J_{V,V}: F(V\otimes V) \cong F(V) \otimes F(V)$ the tensor structure of $F$.  Using  $J$ repeatedly, we have
$$F(V^{\overrightarrow{\otimes n}}) \cong F(V)^{\overrightarrow{\otimes n}}$$ for each $n\in \mathbb{N}$. Here by definition, $V^{\otimes \overrightarrow{n} }=\left( \cdots\left( \left( V \otimes V\right) \otimes V\right) \cdots \otimes V\right). $
The associative constraint in $\HHd$ is trivial, hence $F(V)^{\overrightarrow{\otimes n}} \cong F(V)^{\otimes n}$. Thus we have $\widetilde{F}(\bigoplus\limits_{n \in \mathbb{N}}V^{\overrightarrow{\otimes n}})\cong \bigoplus\limits_{n \in \mathbb{N}}F(V^{\overrightarrow{\otimes n}})\cong \bigoplus\limits_{n \in \mathbb{N}}F(V)^{\otimes n}$. That is $\widetilde{F}(T_{\omega}(V)) \cong T(F(V))$.
\par  Recall that the finite-dimensional twisted Nichols algebra in $\GG$ are of the form $T_{\omega}(V)/I$ where $V \in \GGd$ and $I$ is the unique maximal graded Hopf ideal in $T_{\omega}(V)$ generated by homogeneous elements of degree greater than or equal to $2$.  $\widetilde{F}$ is exact since $\GG$ is semisimple, then using $\widetilde{F}(T_{\omega}(V)) \cong T(F(V))$, we have
$$\widetilde{F}(T_{\omega}(V)/I) \cong T(F(V))/\widetilde{F}(I),$$ which is finite-dimensional as well.
Here $\widetilde{F}(I)$ is a homogeneous Hopf ideal of degree greater than or equal to $2$ of $T(F(V))$. Note that Nichols algebra generated by $F(V)$ must be of the form $T(F(V))/J$, where $J$ is the unique maximal homogeneous graded Hopf ideal of $T(F(V)) \in \HH$ with degree greater than or equal to $2$.  Hence   $ \widetilde{F}(I) \subset J$ and $T(F(V))/J \subset T(F(V))/\widetilde{F}(I)$.
On the other hand, $\widetilde{F}^{-1}: \HH \longrightarrow \GG$ is inverse of $\widetilde{F}$, which is an exact tensor functor. Hence
$$\widetilde{F}^{-1}(T(F(V))/J) \cong T_{\omega}(F^{-1}(F(V))/\widetilde{F}^{-1}(J)\cong T_{\omega}(V)/\widetilde{F}^{-1}(J) \supseteq T_{\omega}(V)/I.$$
So $\widetilde{F}^{-1}(J)\subseteq I$, combining $\widetilde{F}(I) \subseteq J$ implies $\widetilde{F}(I)=J$, which leads to $F(T_{\omega}(V)/I) \cong T(F(V))/J$ and the proof is done. 
\end{proof}
	\subsection{Proof of Theorem \ref{thm1.3}}
 The situation considered in the \cite{huang2024classification} is the following: 	Let $G=Z_{2} \times Z_{2} \times Z_{2}=\left\langle e\right\rangle \times \left\langle f\right\rangle \times \left\langle g\right\rangle $ and $\omega$ the nontrivial $3$-cocycle on $G$:
	\begin{equation}		\omega(e^{i_1}f^{i_2}g^{i_3},e^{j_1}f^{j_2}g^{j_3},e^{k_1}f^{k_2}g^{k_3})=(-1)^{k_1j_2i_3} \label{4.1}
	\end{equation}
 for $ 0 \leq i_1,j_1,k_1 < 2, \ 0 \leq i_2, j_2, k_2 < 2, \ 0\leq i_3, j_3, k_3 <2$.  
 Let $V_1,V_2,V_3 \in \GG $ be three $2$-dimensional pairwise non-isomorphic simple objects in $\GG$ such that $\operatorname{deg}(V_1)=e$, $\operatorname{deg}(V_2)=f$, $\operatorname{deg}(V_3)=g$. Proposition 4.1 in \cite{huang2024classification}  just states that $\mathcal{B}(V_1\oplus V_2\oplus V_3)\in \GG$ must be infinite-dimensional. Our work on categorical Morita equivalence can give us a new proof now.

\begin{prop} \label{prop4.6}
	The category $\GGd$ is braided fusion equivalent to $\Dff$.
\end{prop}

\begin{proof}
	Existence of the finite group $H$ is immediately by Theorem \ref{thm1.2}. Explicitly, 
	take $A=\left\langle e\right\rangle $, $K=\left\langle f\right\rangle  \times \left\langle g\right\rangle $, $F=1$ in Lemma \ref{Lem3.1}. Let
	$$ \hat{F}(f^{i_2}g^{i_3},f^{j_2}g^{j_3})=\chi^{j_2i_3},$$
	where $\chi \in \widehat{A}$ is primitive such that $\chi(g_1)=-1$. 
	%Hence $$\hat{F}(f^{i_2}g^{i_3},f^{j_2}g^{j_3})(e^{k_1})=\zeta_{(m_1,m_2,m_3)}^{a_{123}k_1j_2i_3}$$
	and let $\varepsilon\equiv 1$, 
	%then$$\hat{F} \wedge F(f^{i_2}g^{i_3},f^{j_2}g^{j_3},f^{k_2}g^{k_3},f^{l_2}g^{l_3})=\hat{F}(f^{i_2}g^{i_3},f^{j_2}g^{j_3})(1)=1=\delta_K \varepsilon$$
	%Direct computation shows:
%	$$ \omega((e^{i_1},f^{i_2}g^{i_3}),(e^{j_1},f^{j_2}g^{j_3}),(e^{k_1},f^{k_2}g^{k_3}))=(-1)^{k_1j_2i_3}$$
%	$$\hat{\omega}((f^{i_2}g^{i_3},\chi^{i_1}),(f^{j_2}g^{j_3},\chi^{i_2}),(f^{k_2}g^{k_3},\chi^{i_3})=\chi^{i_1}(1)=1$$
	By Theorem \ref{thm1.2}, $\text{Vec}_G^{\omega}$ and $\text{Vec}_{{\widehat{Z_{2}}} \underset{\hat{F}}{\rtimes} ({Z_{2}\times Z_{2}} ) }$ are categorical Morita equivalent. Let $H={\widehat{Z_{2}}} \underset{\hat{F}}{\rtimes} {Z_{2}\times Z_{2}} $.  Actually, $H$ is isomorphic to $D_8$ since $H$ has such a presentation 
	\begin{align*}
	<(1, (f,g)),\ (1,(f,1)) \mid   (1, (f,g))^4&= (1,(1,1))=     (1,(f,1))^2,\\ &\  (1,(f,1))\cdot   (1, (f,g)) \cdot (1,(f,1))= (1, (f,g))^{-1}>.
	\end{align*}
	Hence 
	$$\GGd\simeq \mathcal{Z}(\text{Vec}_G^{\omega}) \simeq  \mathcal{Z}(\text{Vec}_{D_8}) \simeq \Dff$$ as braided fusion category.
\end{proof}
%By Lemma \ref{lem4.12} and Propostion \ref{prop4.7}, we may extend $\widetilde{F}$ to a tensor equivalence $F: \GG \rightarrow \Df$ which maps twisted Nichols algebra to Nichols algebra.
%Hence, to show $B(V)$ is infinite-dimensional, it suffices to show $F(B(V))$ is infinite-dimensional in $\Df$.
%\par 

%Now we are going to use Cartan matrix and the work of Heckenberger to prove the main theorem. Although it is difficult to compute the Cartan matrix of $B(F(V_1)\oplus F(V_2) \oplus F(V_3))\in \HH$ by direct computation, we could achieve it via calculating in $\GG$. 
Now we are going to prove Theorem \ref{thm1.3}.
\begin{proof}[Proof of Theorem \ref{thm1.3}]
Let $V=(V_1,V_2,V_3)$ be the $3$-tuple. Since $V_1, V_2, V_3$ are pairwise non-isomorphic, simple and  $F: \GGd \longrightarrow \Dff$ is a braided fusion equivalence  then $F(V_1)$, $F(V_2)$ and $F(V_3)$ are pairwise nonisomorphic and simple. 
\par Suppose  $\mathcal{B}(V) \in \GGd$ is finite-dimensional, then $F(\mathcal{B}(V))$ should be a finite-dimensional Nichols algebra of rank $3$ in $\Dff$ since $F$ maps  Nichols algebra in $\GGd$ to usual Nichols algebra by Proposition \ref{prop4.7}. But all Nichols algebra generated by three pairwise nonisomorphic simple Yetter-Drinfeld module over $D_8$ are infinite-dimensional by Theorem \ref{thm5.7}. This is a contradiction, so $\mathcal{B}(V)$ is infinite-dimensional.
\end{proof}
%\begin{cor}\label{cor4.18}
	%The corresponding coquasi-Hopf algebra $H:=\mathcal{B}(V)\# \vmathbb{k}G$ is infinite-dimensional. 
%\end{cor}
%\begin{proof}
	%It's immediately since $\mathcal{B}(V)$ is infinite-dimensional.
%\end{proof}
\begin{rmk}
	\rm It is not hard to see the our method can be applied to more general situation, for example, the case: $G$ is of the form $Z_{m_1}\times Z_{m_2}\times Z_{m_3}=\left\langle g_1\right\rangle  \times \left\langle g_2\right\rangle  \times \left\langle g_3\right\rangle $ and $\omega$ be a nontrivial $3$-cocycle on $G$:
	\begin{align*}
		& \omega\left( g_1^{i_1}  g_2^{i_2}g_3^{i_3}, g_1^{j_1} g_2^{j_2} g_3^{j_3}, g_1^{k_1} g_2^{k_2}g_3^{k_3}\right)  = \zeta_{\left(m_1, m_2, m_3\right)}^{a_{123} k_1 j_2 i_3}.
	\end{align*} 
\end{rmk}
 
	\bibliographystyle{alpha}\small
	\bibliography{Nichols-ref}
\hrulefill
	\begin{center}
		Bowen Li, School of Mathematics, Nanjing University,\\ Nanjing 210093, P. R. China\\
		E-mail: DZ21210002@smail.nju.edu.cn
		\\[10pt]
		Gongxiang Liu, School of Mathematics, Nanjing University,\\ Nanjing 210093, P. R. China\\
		E-mail: gxliu@nju.edu.cn
\end{center}
\end{document}